\documentclass[11pt,a4paper,twoside]{article}
\usepackage[utf8]{inputenc}
\usepackage[english]{babel}

\usepackage{comment}

\usepackage[top=1in, bottom=1in, left=1.25in, right=1.25in]{geometry}

\usepackage{amssymb}
\usepackage{amsmath}
\usepackage{amsthm}

\usepackage{enumitem}

\usepackage{graphicx}

\newtheorem{theorem}{Theorem}[section]
\newtheorem*{theorem*}{Theorem}
\newtheorem{proposition}[theorem]{Proposition}
\newtheorem{lemma}[theorem]{Lemma}
\newtheorem{corollary}[theorem]{Corollary}

\theoremstyle{definition}
\newtheorem{definition}[theorem]{Definition}
\newtheorem{remark}[theorem]{Remark}

\newenvironment{myproof}[2] {\emph{Proof of {#1} {#2}.}}{\hfill$\square$}

\DeclareMathOperator{\reals}{\mathbb{R}}
\DeclareMathOperator{\naturals}{\mathbb{N}}

\DeclareMathOperator{\stab}{Stab}
\DeclareMathOperator{\diam}{diam}
\DeclareMathOperator{\NW}{NW}
\DeclareMathOperator{\Per}{Per}

\newcommand{\dd}{\mathrm{d}}

\title{Topology of horocycles on geometrically finite nonpositively curved surfaces}

\author{Sergi Burniol Clotet\\
	IMERL, Udelar, Av. Julio Herrera y Reissig 565, Montevideo\\
	(email: sergi.burniol@gmail.com)}

\begin{document}

\maketitle

\begin{abstract}
	We study the closure of horocycles on rank $1$ nonpositively curved surfaces with finitely generated fundamental group. Each horocycle is closed or dense on a certain subset of the unit tangent bundle. 
	In fact, we classify each half-horocycle in terms of the associated geodesic rays. 	
	We also determine the nonwandering set of the horocyclic flow and characterize the surfaces admitting a minimal set for this flow. 
\end{abstract}

Key words: horocyclic flow, nonpositive curvature, geometrically finite, orbit closures

\section{Introduction}

Let $M$ be a complete surface with nonpositive curvature. We say that $M$ is geometrically finite if the fundamental group is finitely generated. We assume that $M$ is nonelementary, which means that the fundamental group is not cyclic, or equivalently, that the surface topologically is not a cylinder. It is also assumed that the curvature is not identically zero, so $M$ is a rank $1$ surface. The goal of this article is to study the topological properties of the horocyclic flow $h_s$ on the unit tangent bundle $T^1M$. At the beginning of the next section, we recall the definition of this flow.

In \cite{Eberlein79}, Eberlein studied the geometry of ends of nonpostively curved surfaces. When $M$ is geometrically finite there are finitely many ends and each one is collared, which means that they have a neighborhood homeomorphic to a cylinder. He suggested two classifications, one metric and the other dynamical, of the collared ends into four types: simple parabolic, exceptional parabolic, cylindrical and expanding. Link, Picaud and Peigné proved that the two classifications are equivalent \cite{LinkPeignePicaud}. In strictly negative curvature, a simple parabolic end is commonly referred to as a cusp and the expanding end as a funnel. The cylindrical end has a neighborhood isometric to a flat cylinder $S^ 1\times \reals_+$. The exceptional parabolic end has a behavior between a cylinder and an exceptional parabolic end.

Let $X$ be the universal cover of $M$ with the lifted metric, and let $\Gamma $ be the group of covering transformations, which is identified to the fundamental group of $M$. We recall that $X$ can be compactified by adding a visual boundary $\partial X$ (see Section \ref{Sec:preliminaries} for the definition and basic properties). The limit set $\Lambda $ of $\Gamma$ is the set of accumulation points in $\partial X$ of an orbit $\Gamma p$, $p\in X$. It is also the smallest closed $\Gamma$-invariant subset of $\partial X$ \cite[Theorem 2.8]{Ballmann82}. We define the set $\Omega^+$ of vectors in $T^ 1M$ which have a lift in $T^1 X$ that points in the positive direction to the limit set $\Lambda$.

Our main result establishes what is the topology of any horocyclic orbit in $T^1 M$. A geodesic ray $c:\reals_+\to M$ is minimizing if the distance between any of its points is realized by the same geodesic ray $c$. We say that $c$ is eventually minimizing if for some $T>0$, $c $ is minimizing on $[T,+\infty)$.

\begin{theorem}\label{main_theorem}
	Let $M$ be a nonelementary geometrically finite surface with nonpositive curvature. For $v\in T^ 1M$,
	\begin{enumerate}
		\item if the geodesic ray $g_{\reals_+}v$ converges to a parabolic or cylindrical end and is eventually minimizing, then $h_{\reals}v$ is periodic, hence compact,
		
		\item if the geodesic ray $g_{\reals_+}v$ converges to an exceptional parabolic or cylindrical end and is not eventually minimizing, then $\overline{h_{\reals}v} = h_{\reals}v \cup \Omega^ + $,
		
		\item if the geodesic ray $g_{\reals_+}v$ converges to an expanding end, then $h_{\reals}v$ is closed and also converges to that end,
		
		\item if the geodesic ray $g_{\reals_+}v$ does not converge to an end, then $\overline{h_{\reals}v} = \Omega^ + $.
	\end{enumerate}
\end{theorem}

This result generalizes the work of Hedlund on finite volume hyperbolic surfaces \cite{Hedlund36} and Dal'bo on geometrically finite manifolds with strictly negative curvature \cite{Dalbo00}. Eberlein characterizes dense horospheres on a class of nonpositively curved manifolds where every unit tangent vector is nonwandering under the geodesic flow \cite{Eberlein73a}. Ergodic properties of the horocyclic flow on nonpositively curved surfaces were also studied in \cite{Burniol21,Burniol23}. Under the more general hypothesis of our theorem, Link, Peigné and Picaud show that the horocyclic flow in restriction to $\Omega^+$ is transitive i.e. it has a dense orbit \cite{LinkPeignePicaud}. In fact, they show that horocyclic orbits of periodic vectors for the geodesic flow are always dense in $\Omega^+$.

The vectors satisfying any of the conditions of the theorem above also admit a nice description in terms of their end points at infinity and their relation with the limit set $\Lambda$ and the group $\Gamma$. The phenomenon presented by horocyclic orbits of item $2$ does not occur in strictly negative curvature. We will see that these orbits are nonwandering, but still they accumulate to the set $\Omega^+$.

To prove this classification, it is convenient to understand the behavior of the positive and the negative half-horocycles $h_{\reals_+}(v)$ and $h_{\reals_-}(v)$, since it is not always the same. We know the following:

\begin{itemize}
	\item in item $2$, one half-horocycle accumulates to $\Omega^+$ and the other converges to the corresponding end,
	\item in item $4$, if the vector $v$ generates a closed geodesic which bounds an expanding end, or is asymptotic to such a geodesic, then one half-horocycle is dense in $\Omega^+$ and the other converges to the expanding end,
	\item if $v$ is any other vector in item $4$, each half-horocycle is dense in $\Omega^ +$.
\end{itemize} 

In strictly negative curvature, the study of half-horocycles was carried out by Schapira \cite{Schapira11}, with the same conclusion for vectors in item $4$. Our methods are partly based on her work.

A closed $h_s$-invariant nonempty subset $\mathcal{M}$ of $T^1 M$ is called minimal if there is no proper subset of $\mathcal{M}$ with that property. Periodic or closed horocycles are trivial minimal sets. A limit point $\xi\in \Lambda$ is called \emph{horocyclic} if there exist a point $p$ in $X$ and a sequence $\gamma_n$ in $\Gamma$ such that $\gamma_n p$ converges to $\xi$ and the the sequence $\gamma_n p$ enters every horoball centered at $\xi$ (see Section \ref{Sec:preliminaries} for more details).
One relevant consequence of our classification is the following.

\begin{corollary}\label{minimal_subsets}
	The horocyclic flow of $M$ has a nontrivial minimal set if and only if all the points in the limit set $\Lambda$ are horocyclic. If it exists, the nontrivial minimal set is $\Omega^+$.
\end{corollary}

A vector $v\in T^1 M$ is nonwandering for the horocyclic flow $h_s$ if for every open subset $U$ of $T^1M$ containing $v$ the set $$\{s\in \reals : h_s(U)\cap U\not = \emptyset \}$$ is unbounded. The vectors satisfying the property of item $1$ in Theorem \ref{main_theorem}  are nonwandering, because they are periodic, but they are not in the set $\Omega^+$ in general. We show:

\begin{theorem}\label{nonwandering_set}
	The nonwandering set of the horocyclic flow is the union of $h_s$-periodic vectors and the set $\Omega^+$.
\end{theorem}

\section{Preliminaries}\label{Sec:preliminaries}

\subsection{Boundary at infinity and horocycles}

Let $X$ be a simply connected nonpositively curved rank $1$ surface. 
We say that two geodesic rays $c_1,c_2:\reals_+\to X$ are \emph{asymptotic} if there exists a constant $C>0$ such that
$$d(c_1(t),c_2(t))\le C $$ for all $t\ge0$. The \emph{boundary at infinity} $\partial X  $ is defined as the set of equivalence classes of asymptotic rays. For a full geodesic $c:\reals \to X$, let $c(+\infty)$ (resp. $c(-\infty)$) denote the class of the positive (resp. negative) ray of $c$. For $v\in T^1X$, we denote by $v_+$ (resp. $v_-$) the asymptotic class of the positive (resp. negative) geodesic ray generated by $v$. We say that $v_+$ is the end point of $v$, and that $v$ points to $v_+$. For every $x\in X$ and every $\xi\in \partial X$, there is a unique vector $V(x,\xi)\in T^1 _xX$ pointing to $\xi$. 

We equip $\overline{X}:=X\cup \partial X$ with the cone topology \cite{EberleinONeill73,ballmannlectures}, which makes this space compact. The boundary at infinity is homeomorphic to a circle, which we orient in the counterclockwise direction. Given two points $\eta$ and $\xi$ in $\partial X$, we denote the set of points on the image of the curve that goes from $\eta $ to $\xi$ in the counterclockwise direction by $[\eta,\xi]$, which is referred to as a \emph{closed interval}. Similarly, we define (half-)open intervals. 

We recall that the \emph{Busemann cocycle} $\beta_\xi: X\times X \to \reals$ at $\xi\in \partial X$ is, for $p,q\in X$,
$$
\beta_\xi(q,p)=\lim_{z\to \xi}d(z,q)-d(z,p).
$$
For fixed $p$, $\beta_\xi(\,\cdot\,,p)$ is said to be a \emph{Busemann function}. 
The \emph{horocycle} with center $\xi\in \partial X$ passing through $p\in X$ is the set 
$$
H_\xi(p)=\{q\in X, \beta_\xi(q,p)=0\}. 
$$ The \emph{closed horoball} bounded by $H_\xi(p)$ is the set 
$$
B_\xi(p)=\{q\in X, \beta_\xi(q,p)\le 0\}.
$$

\begin{proposition}\label{properties_Busemann}\cite{HeintzeImHof}
	For every $p\in X$ and $\xi\in \partial X$, the function $\beta_\xi(\cdot, p)$ is convex and of class $C^2 $. Moreover, the function $(\xi,q,p)\mapsto \beta_{\xi}(q,p)$ is continuous.
\end{proposition}

Therefore, horocycles are simple $C^ 2$ curves of $X$. Let $\pi$ be the canonical projection from $T^1X$ to $X$. For $v\in T^1 X$, the stable horocycle $H^s(v)$ of $v$ is the set of unit vectors with basepoint on $H_{v_+}(\pi(v))$ normal to this curve and directed towards the horoball $B_{v_+}(\pi(v))$. All the vectors in $H^s(v)$ point to $v_+$. The horocyclic flow $h_s:T^1 X\to T^1X$ is obtained as the parametrization of the horocycles on $T^1 X$ by the arc-length on $X$, where the horocycles are also oriented in the counterclockwise direction to match the orientation on the boundary at infinity (Figure \ref{Fig:orientation}). By Proposition \ref{properties_Busemann}, the resulting flow is continuous. 

\begin{figure}
	\centering
	\includegraphics{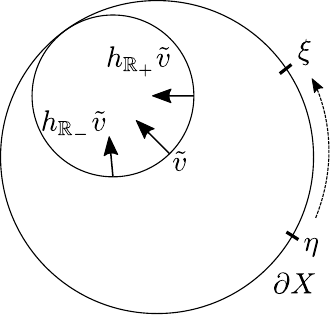}
	\caption{Orientation of the boundary at infinity and the horocycles.}
	\label{Fig:orientation}
\end{figure}

The following is known regarding the relation between horocycles and the boundary at infinity:

\begin{proposition}\label{horcycle_infinity} \cite[Lemme I.7]{LinkPeignePicaud}
	For $p,q\in X$, we set $$H_\xi(\infty)=\{\eta\in \partial X, \, t\mapsto \beta_\xi(\sigma_{p,\eta}(t),q) \text{ is bounded above}\}.$$
	\begin{enumerate}
		\item $H_\xi(\infty)$ does not depend on $p$ and $q$.
		\item $\overline{B_\xi(p)}^{\bar{X}}\cap \partial X = H_\xi(\infty).$
		\item $H_\xi(\infty)$ is a closed interval $I=[\xi_1,\xi_2]$ in $\partial X$ containing $\xi$.
		\item If $\xi_1\not =\xi_2$, the only pair of points in $H_\xi(\infty)$ that may be joined by a geodesic is $\xi_1, \xi_2$. If this two points are joined by a geodesic, then this geodesic bounds a flat half-plane.
		\item If $\eta\in \mathring I$, then $\lim_{t\to +\infty} \beta_\xi(c_{p,\eta}(t),q)=-\infty$, where $c_{p,\eta}$ is the geodesic ray starting at $p$ and pointing to $\eta$.
		
		\item For $p\in X$, the horocycle $H_\xi(p)$ accumulates exactly at the points $\xi_1$ and $\xi_2$. 
	\end{enumerate}
\end{proposition}

\subsection{Joining points at infinity}

We have mentioned that it is always possible to join a point in $X$ to a point in $X\cup \partial X$ by a unique geodesic. On the one hand, there may be more than one geodesic joining two different points in the boundary.

\begin{theorem} (Flat Strip Theorem)
	If there exist two geodesics $c_1, c_2$ with the same end points in $\partial X$, then $c_1 $ and $c_2$ bound a flat strip isometric to $[0,a]\times \reals$, where $ a $ is the distance between $c_1$ and $c_2$.
\end{theorem}

On the other hand, it is not always possible to join two distinct points at infinity by a geodesic. Examples of this are flat half-planes or flat sectors, or more generally the situation of item 4 in Proposition \ref{horcycle_infinity}. The following lemma gives a sufficient condition to join points at infinity.

\begin{lemma}\cite[Lemma III.3.1]{ballmannlectures} Let $c$ be a geodesic that does bound a flat strip of width $R>0$. Then there exist neighborhoods $U$ and $V$ of $c(+\infty)$ and $c(-\infty)$ in $\partial X$ such that any pair of points in $U\times V$ can be joined by a geodesic $c'$. For any such geodesic $c'$, $d(c',c(0))<R$. 
\end{lemma}

Finally, we recall some equivalent conditions to joining points at infinity.

\begin{proposition}\label{joining}
	Let $\xi,\eta\in \partial X $ be distinct points at infinity. The following are equivalent:
	\begin{enumerate}
		\item there is a geodesic joining $\xi$ to $\eta$,
		
		\item there exist sequences $x_n$ ans $y_n$ in $X$ converging respectively to $\xi$ and $\eta$, a point $p\in X$ and a constant $C$ such that the geodesics $c_n$ joining $x_n$ to $y_n$ satisfy
		$$
		d(c_n, p)\le C,
		$$
		
		\item there exist horocycles $H_\xi$ and $H_\eta$ centered at $\xi$ and $\eta$ such that $\# H_\xi \cap H_\eta \ge 2$,
		
		\item if $h_\xi$ is a Busemann function centered at $\xi$ and $h_\eta$ is a Busemann function centered at $\eta$, then $h_\xi + h_\eta $ assumes its minimum.
		
		Moreover, if there is no geodesic joining $\xi$ to $\eta $, $H_\xi(\infty)\cap H_\eta(\infty) \not =\emptyset.$
		If there is a geodesic joining $\xi$ to $\eta $ and $H_\xi(\infty)\cap H_\eta(\infty) \not =\emptyset$, then the geodesic bounds a flat half-plane.

	\end{enumerate}
\end{proposition}

\begin{proof}
	The equivalence between 1 and 2 is standard. The equivalence with 3 and the final statement are shown in \cite[Corollaire I.8]{LinkPeignePicaud}. The equivalence with 4 can be found in the Proof of Proposition 9.35 in \cite{BridsonHaefliger}
\end{proof}

\subsection{Tits metric}

It is not essential in this article, but it can be useful to briefly recall the definition of the Tits distance on the visual boundary $\partial X$. Given $\xi, \eta \in \partial X$, we first define the \emph{angular distance} between $\xi $ and $\eta$ as 
$$\angle (\xi,\eta)=\sup_{p\in X} \angle_p(\xi,\eta).$$ 
The \emph{Tits distance} is defined as the length metric of $(\partial X, \angle)$ and has values in $[0,+\infty]$.

\begin{proposition} \cite[4.10]{ballmannnonpositivebook} 
	 The set $H_\xi(\infty)$
	is equal to the closed ball of radius $\pi/2$ and center $\xi$ in the Tits distance. 
\end{proposition}

\subsection{Isometries}

We now study the isometries of $X$. Given an isometry $\gamma$, we consider the displacement function $d_\gamma:X\to \reals$ defined by $d_\gamma(p)=d(p,\gamma p)$. We denote the infimum of this function by $l_\gamma$.

\begin{proposition}\label{isometries}
	There are four possibilities:
	\begin{enumerate}
		\item $\gamma$ is elliptic if it has a fixed point on $X$.
		\item $\gamma$ is axial if $d_\gamma $ assumes its infimum $l_\gamma$ and $l_\gamma>0$;  then $\gamma$ leaves invariant the geodesic through $p$ and $\gamma p$ and acts on it as a translation by $l_\gamma$. Such a geodesic is called an axis. If there are two axis, then they bound a flat strip. The isometry fixes at least the two end points of their axis.
		\item $\gamma$ is simple parabolic if $d_\gamma $ does not assume its infimum $l_\gamma$ and $l_\gamma=0$; then $\gamma$ fixes a single point $\xi$ at infinity and leaves invariant all the horocycles centered at this point.
		\item $\gamma$ is exceptional parabolic if $d_\gamma $ does not assume its infimum $l_\gamma$ and $l_\gamma>0$; then there exists a unique $\xi\in \partial X$ such that $\gamma$ leaves invariant every horocycle centered at $\xi$; and $\gamma$ fixes the nontrivial interval $H_\xi(\infty)$ point by point.
	\end{enumerate}
\end{proposition}

If $\gamma$ is a parabolic isometry, the point $\xi$ given by the previous preposition is called the \emph{center} of $\gamma$.

\subsection{Limit set}
Let $\Gamma$ be a discrete group of isometries of $X$.
\begin{definition}
	The limit set $\Lambda$ of $\Gamma$ is the set of accumulation points in $ \partial X$ of an orbit $\Gamma p$, $p\in X$, i.e. 
	$$\Lambda = \overline{\Gamma p} \cap \partial X.$$
\end{definition}

We distinguish different types of limit points.

\begin{definition}
	Let $\xi\in \Lambda$ be a limit point. 
	\begin{itemize}
		\item We say that $\xi$ is \emph{radial} if there exists a sequence $\gamma_n$ in $\Gamma$ such that, for some $p\in X$ and some geodesic ray $c$ pointing to $\xi$, $\xi =\lim_{n}\gamma_n p$ and $d(\gamma_np,c)$ is bounded for all $n\in\mathbb{N}$.
		
		\item We say that $\xi$ is \emph{horocyclic} if there exists a sequence $\gamma_n$ in $\Gamma$ such that, for some $p\in X$, $\xi =\lim_{n}\gamma_n p$ and $\lim_n\beta_\xi(\gamma_np,p)=-\infty$.
		
		\item We say that $\xi$ is \emph{simple parabolic} if it is the fixed point of a simple parabolic isometry.
		
		\item We say that $\xi$ is \emph{exceptional parabolic} if it is the fixed point of an exceptional parabolic isometry.
	\end{itemize}
\end{definition}

Being a simple parabolic limit point is the same as being the center of a simple parabolic isometry, because the fixed point is unique.

\begin{lemma}\label{joining_limit}
	Two points $\eta $ and $\xi$ in $\partial X$ that are not joined by a geodesic bound an interval of $\partial X$ which is disjoint from the limit set $\Lambda$. In particular, for any $\zeta \in\partial X$, the limit set does not intersect the interior of $H_\zeta(\infty)$.
\end{lemma}

\begin{proof}
	Assume that there is no geodesic joining $\eta $ to $\xi$. Let $p$ be any point in $X$. Necessarily, the non-oriented angle $\alpha$ from $p$ between $\eta $ and $\xi$ is stricly less than $\pi$. Let $D$ be the sector of $X$ delimited by the geodesics $c_{p,\eta }$ and $c_{p,\xi}$ with angle $\alpha$. For $t>0$, we consider the triangle $D_t$ with vertices $p$, $c_{p,\eta}(t)$ and $c_{p,\xi}(t)$. Let $\beta_t$ and $\gamma_t$ the angles of $D_t$ at $c_{p,\eta}(t)$ and $c_{p,\xi}(t)$, respectively. By the Gauss-Bonnet formula, the total curvature of $D_t$ is 
	\[
	-\int_{D_t } K \dd A = \pi -\alpha -\beta_t -\gamma_t \le \pi -\alpha .
	\]
	Proposition \ref{joining} implies that $D$ can be obtained as the union of $D_t$ for $t>0$. By monotone convergence, the total curvature of $D$ is the limit of the total curvatures of $D_t$ when $t$ tends to infinity. We conclude that the total curvature of $D$ is bounded.
	
	
	The intersection of the closure of $D$ in $X\cup \partial X$ with $\partial X$ is a closed interval of $\partial X$. We take $I$ as the interior of this interval. Let $q\in X$ be a point where the curvature is negative and let $\varepsilon>0$. Since the total curvature of $D$ is bounded, there is a finite number of $\gamma\in \Gamma$ such that $\gamma B(q,\varepsilon) \subset D$. This implies that no point in $I$ is an accumulation point of $\Gamma q$.
	
	For the second part, since no two points in the interior of $H_\zeta(\infty)$ can be joined by a geodesic, this set cannot intersect $\Lambda$.
\end{proof}

Let us describe more precisely exceptional parabolic limit points. By Proposition \ref{isometries}, given an exceptional parabolic isometry $\gamma$ with center $\xi$, the fixed point set coincides with the interval $H_\xi(\infty)=[\xi_1,\xi_2]$. Both points $\xi_1$ and $\xi_2$ are exceptional limit points, in fact, they are the limits of $\gamma^n p $ and $\gamma^{-n} p$, $p\in X$, when $n$ tends to infinity. Moreover, by Lemma \ref{joining_limit}, they are the only limit points in the interval $[\xi_1,\xi_2]$. This shows that every exceptional parabolic isometry $\gamma$ has exactly two associated exceptional limit points.
  We ignore wether the center $\xi$ of the isometry $\gamma$ can coincide with one of the points $\xi_1$ and $\xi_2$.


An extremity of an interval in $\partial X\setminus \Lambda $ corresponding to a cylindrical or an expanding end is a radial limit point.

\begin{lemma}\label{intersection_types_limit}
	\begin{enumerate}
		\item Every radial point is horocyclic. 
		\item Limit points that are the center of a parabolic isometry are not horocyclic.
		\item Exceptional parabolic limit points which are not the center of the associated parabolic isometry are horocyclic but not radial.
	\end{enumerate}
\end{lemma}

\begin{proof}
	
	\begin{enumerate}
		\item Let $\gamma_n$ be a sequence in $\Gamma$ such that $\lim_n \gamma_n p=\xi$ and $d(\gamma_n p, c)\le R$, where $c$ is the geodesic ray from $p$ to $\xi$. 
		Let $c(t_n)$ be the point in $c$ closest to $\gamma_n p$. Since $\gamma_n p$ converges to $\xi$, $t_n$ goes to $+\infty$. We have
		$$\beta_\xi(\gamma_n p , p)= \beta_\xi(\gamma_n p , c(t_n)) + \beta_\xi( c(t_n), p)\le R -t_n, $$ which shows that $ \lim_n \beta_\xi(\gamma_n p , p)=-\infty$.
		
		\item Let $\xi $ be the center of a parabolic isometry $\gamma_0$. By a result of Eberlein (see Proposition \ref{parabolic_ngh}), there exists a horoball $B$ centered at $\xi$ which satisfies for every $\gamma \in \Gamma$ either $\gamma B=B$ or $\gamma B\cap B=\emptyset $. This implies that for any point $p$ in the boundary of $B$, $\beta_\xi(\gamma p, p )\ge 0$. So $\xi$ cannot be horocyclic.
		
		\item Let $\xi$ be an exceptional parabolic limit point. Necessarily, $\xi$ is the extremity of an open interval $I$ in $\partial X\setminus \Lambda$ and there exists an exceptional parabolic isometry $\gamma _0$ which fixes every point in the closure of the interval $I$. The center of $\gamma_0$ is different from $\xi$ by hypothesis. As before, there exists a horoball $B$ centered at $\xi$ which satisfies for every $\gamma \in \Gamma$ either $\gamma B=B$ or $\gamma B\cap B=\emptyset $. Let $p$ be any point in the boundary of $B$ an consider the horoball $B'$ centered at $\xi$ whose boundary contains $p$. The horoball $B'$ intersects the horocycle $\partial B$ in a half-horocycle by Propostion \ref{joining}.
		
		Up to taking the inverse we know that $\gamma_0^ n p  $, $n\in \naturals$, stays in the horoball $B'$ and $\lim_n \gamma_0^ n p = \xi$. We have
		$$\beta_\xi(\gamma_0^ n p , p )= n \beta_\xi(\gamma_0 p , p )
		$$
		because $\gamma_0$ fixes $\xi$.
		Since $\gamma_0$ does not fix the horocycles centered at $\xi$, we must have $\beta_\xi(\gamma_0 p , p )<0$. This implies that $\lim _n \beta_\xi(\gamma_0^ n p , p )=-\infty$, showing that $\xi$ is horocyclic.
		
		Let us explain why $\xi$ is not radial. The orbit $\Gamma p$ is contained in the union of the boundary of $B$ with its complement. Let $c$ be the geodesic ray starting at $p$ pointing to $\xi$. This ray is included in the horoball $B$. We will show that the distance between $c(t)$ and $\partial B$ goes to infinity, when $t\to +\infty$ in Lemma \ref{distance_geod_horo} below. The same will be true for any geodesic ray asymptotic to $c$. So there is no sequence $\gamma_n p$ converging to $\xi$ which stays within a bounded distance from a ray pointing to $\xi$.  
	\end{enumerate}

\end{proof}

\begin{lemma}\label{distance_geod_horo}
	Let $\gamma $ be an exceptional parabolic isometry with center $\xi$. Let $H_{\xi}(\infty)=[\xi_1,\xi_2]$. Let $c$ be a geodesic ray pointing to $\xi_1$ and let $H$ be the horocycle centered at $\xi$ and passing through $c(0)$. Then
	$$\lim_{t\to +\infty} d(c(t), H) =+\infty.
	$$
\end{lemma}

\begin{proof}
	We observe that $c$ stays inside the horoball bounded by $H$ and does not intersect $H$ for positive time. Otherwise, $c$ and $H$ would coincide and they would bound a flat-half plane. 
	
	Consider the convex function $f$ defined by $f(t):=\beta_\xi (c(t), c(0) )$. We observe that 
	$$
	d(c(t),H)=|f(t)|
	$$
	 It satisfies $f(0)=0$ and $f'(0)<0$. By the observation made above, $f$ does not vanish again for $t>0$. By the convexity it is non-increasing. Let $a =\lim_{t\to +\infty} f(t)\in [-\infty,0)$. We show that $a=-\infty$ to finish the proof. Assume that $a$ is finite and we reason by contradiction. We consider the horocycle $H'$ centered at $\xi$ such that $\beta_\xi(H',H)= a$. Then $d(c(t), H')= f(t)-a\to 0$ when $t\to +\infty$. In the quotient $M$, $H'$ projects to a periodic horocycle $\overline{H'}$ and $c$ projects to a geodesic ray that acumulates to $\overline{H'}$. By the continuity of the geodesic flow, $\overline{H'}$ is also a geodesic. But this can only happen if $\overline{H'}$ bounds a flat cylinder. This contradicts the fact that $\overline{H'}$ bounds an exceptional parabolic end.

\end{proof}

\subsection{Gromov product}

We introduce the Gromov product for certain pairs of points at infinity. We will use it as notation in some proofs.

\begin{definition}
	Let $\eta$ and $\xi$ be two points in $\partial X$ that are joined by a geodesic $c$. Choose any point $p_{\eta, \xi} $ on $c$. We define the \emph{Gromov product} of $\xi $ and $\eta$ at $x\in X$ as
	$$ \langle\eta , \xi \rangle_x :=\frac{1}{2}( \beta_{\xi}(x,p_{\eta,\xi})+\beta_{\eta}(x,p_{\eta,\xi})). $$
\end{definition}

One can show that the previous definition does not depend on the choice of the geodesic $c$ nor the choice of the point $p_{\eta,\xi}$ on $c$. A continuous Gromov product with values in $\reals\cup\{+\infty\}$ can be defined for any pair of points in the limit set $\Lambda$ \cite[Théorème B]{LinkPeignePicaud}.

\subsection{Ends of nonpositively curved surfaces}

\begin{definition}
	An \emph{end} of a surface $M$ is a map $\varepsilon$ which, to each compact subset $K$ of $M$, associates a connected component of $M\setminus K$ with the property that, for any two compact subsets $K_1\subset K_2$, $\varepsilon(K_2)\subset \varepsilon(K_1)$.
	A subset $U\subset M$ is a \emph{neighborhood of} $\varepsilon$ if, for some compact subset $K$, $U$ contains $\varepsilon(K)$.
	A sequence of points or curves $(c_n)_n$ \emph{converges to the end} $\varepsilon$ if, for any neighborhood $U$ of $\varepsilon$, $c_n$ is contained in $U$ for $n$ large enough.
\end{definition}

We notice that we can associate an end $\varepsilon$ of $M$ to each divergent geodesic ray $c:\reals_+\to M$ by choosing as $\varepsilon(K)$ the connected component of $M\setminus K$ where $c([T,\infty))$ is contained for $T$ large enough. We will also say that $c$ converges to the end $\varepsilon$.

The simplest topology that an end can have is that of a cylinder.

\begin{definition}
	An end $\varepsilon$ is \emph{collared} if it has a neighborhood homeomorphic to a cylinder $S^1 \times (0,+\infty)$.
\end{definition}

Eberlein described the metric of collared ends.

\begin{theorem} \label{coordinates-nbgh}
	Let $\varepsilon$ be a collared end of $M$. Then there exist an open neighborhood $U$ of $\varepsilon$ with local $C^ 1$ coordinates
	$
	(r,\theta )\in (0,+\infty)\times S^ 1
	$ such that the metric in $U$ has the form 
	\begin{equation}\label{tubular_parametrization}
		dg^ 2=dr^ 2+ G(r,\theta)^ 2 d\theta^ 2,
	\end{equation}
	where $G: (0,+\infty)\times S^ 1\to (0,+\infty)$ is a continuous function with $G(\cdot , \theta)$ twice differentiable and convex, for each $\theta \in S^ 1$.
	
	Moreover, if there exists a sequence of continuous picewise smooth curves $(c_n)_n$ converging to an end $\varepsilon'$ such that $(c_n)_n$ belong to the same nontrivial free homotopy class, then $\varepsilon' $ is collared.
\end{theorem}

In a parametrization of a collared end as above, the $r$-curves are minimizing geodesic rays. The curvature $K$ at the point with coordinates $(r,\theta)$ is equal to \[
K(r,\theta)=- G(r,\theta)^ {-1} \frac{\partial^ 2 G}{\partial r^ 2}(r,\theta).
\]
When the function $G$ does not depend on the coordinate $\theta$, we obtain a surface of revolution. In order to distinguish different behaviors of the ends, we consider the function $L:(0,+\infty)\to (0,+\infty)$ giving the length of the corresponding $\theta$-curve, that is,
\[
L(r)=\int_{S^1 } G(r,\theta)\, \dd \theta,
\]
For an end $\varepsilon$, we denote the set of vectors in $T^1 M$ whose positive geodesic ray converges to $\varepsilon$ by $V(\varepsilon)$.

The following result derives from the equivalence of the two classifications of collared ends proposed by Eberlein, and it was proved by Link, Peigné and Picaud.

\begin{theorem} \cite{LinkPeignePicaud} \label{classification_ends}
	Let $\varepsilon$ be a collared end of a nonelementary nonpositively curved surface $M$. Then $\varepsilon$ satisfies one of the following:

	\begin{enumerate}
		\item (simple parabolic) The set $V(\varepsilon) $ has empty interior and there exists a sequence of smooth curves $(c_n)_n$ converging to $\varepsilon$ with uniformly bounded length. For any parametrization of a neighborhood of $\varepsilon$ of the form (\ref{tubular_parametrization}), the associated function $L$ decreases strictly to $0$.
		\item (exceptional parabolic) The set $V(\varepsilon) $ has nonempty interior but it is not open and there exists a sequence of smooth curves $(c_n)_n$ converging to $\varepsilon$ with uniformly bounded length. For any parametrization of a neighborhood of $\varepsilon$ of the form (\ref{tubular_parametrization}), the associated function $L$ decreases strictly to some $l>0$.
		\item (cylindrical) The set $V(\varepsilon) $ is open and there exists a sequence of smooth curves $(c_n)_n$ converging to $\varepsilon$ with uniformly bounded length. For any parametrization of a neighborhood of $\varepsilon$ of the form (\ref{tubular_parametrization}), the associated function $L$ is constant for $r\ge r_0$.
		\item (expanding) Any sequence of smooth curves $(c_n)_n$ converging to $\varepsilon$ has unbounded length. The set $V(\varepsilon)$ is open. For any parametrization of a neighborhood of $\varepsilon$ of the form (\ref{tubular_parametrization}), the associated function $L$ tends to $+\infty$.
	\end{enumerate}
Moreover, when $\varepsilon $ is parabolic or cylindrical, the $\theta $-curves of the parametrization of a neighborhood of $\varepsilon$ can be choosen to be closed horocycles. If the end $\varepsilon$ is expanding, $U$ can be chosen so that its boundary $\partial U$ is a closed geodesic.
\end{theorem}

 We consider a neighborhood $U$ of the end $\varepsilon$ given by the theorem, which can be thought as the quotient of the universal cover $\tilde U$ of  $U$ by the group generated by an isometry $\gamma$ of infinite order. If $\varepsilon$ is simple (resp. exceptional) parabolic, then $\tilde U$ is a horoball and $\gamma$ is simple (resp. exceptional) parabolic. Conversely, a parabolic isometry produces an end of the same type. In fact: 

\begin{proposition}\cite[Proposition 3.6]{Eberlein79} \label{parabolic_ngh}
	Let $\gamma$ be a parabolic isometry and let $\xi$ be the point in $\partial X$ such that $\gamma$ fixes the horocycles centered at $\xi$. Then, there exists a horoball $B$ centered at $x$ such that, for all $g\in \Gamma$, $gB\cap B =\emptyset $ if $gx\not =x$ and $gB=B$ if $gx=x$.
\end{proposition}

\begin{remark}
	If $\varepsilon$ is the parabolic end associated to the point $\xi$, then $U=B/\stab_\Gamma(x)$ is a neighborhood of $\varepsilon$. A neighborhood of $\varepsilon$ obtained in this way will be called a \emph{standard neighborhood} of $\varepsilon$.
\end{remark}

By Proposition \ref{parabolic_ngh}, there is a correspondence between the conjugacy classes of parabolic isometries in $\Gamma$ and the parabolic ends of the surface $M$, preserving the fact of being simple or exceptional.
 
 For a cylindrical end $\varepsilon$ with neighborhood $U$ given by the theorem, the universal cover $\tilde U$ is a flat half-plane and $\gamma $ is an axial isometry, hence, a translation in the flat-half plan. If $\varepsilon$ is expanding, then $\tilde U$ is a nonflat half-plane and $\gamma$ is an axial isometry. The set $\tilde U$ can be thought as follows:

\begin{proposition}\label{funnel_ngh}
	Let $\varepsilon$ be a cylindrical or expanding end of $M$. Then there exists an axial isometry $\gamma$ which leaves invariant an interval $I$ of $\partial X\setminus \Lambda$ associated to the end $\varepsilon$. If a geodesic $c$ is an axis of $\gamma$, then $c(-\infty )$ and $c(+\infty)$ are the extremities of $I$ and the half-plane $\tilde U$ delimited by $c$ and $I$ satisfies, for all $g\in \Gamma$, $g \tilde U\cap \tilde U =\emptyset $ if $g$ does not fix $c$ and $g\tilde U= \tilde U$ otherwise.
\end{proposition}

\begin{remark}
	The end $\varepsilon$ has a neighborhood of the form  $U=\tilde U/\stab_\Gamma(c)$. This neighborhood is bounded by a closed geodesic unique up to homotopy. A neighborhood of $\varepsilon$ obtained in this way will be called a \emph{standard neighborhood} of $\varepsilon$.
\end{remark}

\subsection{Description of horocycles}

From now on, we consider a nonelementary nonpositively curved surface $M$ and we denote its universal cover by $X$. We equip $X$ with the lifted metric and we let $\Gamma$ denote the covering group of $M$ acting on $X$ by isometries. The goal of this section is to describe the horocycles on $X$. 
We begin with a direct consequence of Lemma \ref{joining_limit}.
\begin{proposition}
	Let $\eta\in \Lambda$ a limit point which does not bound a connected component of $\partial X\setminus \Lambda$. Then any horocycle centered at $\eta$ only accumulates to infinity at $\eta$.
\end{proposition}

To understand the horocycles centered at points which bound a connected component of $\partial X\setminus \Lambda$ or which are not in the limit set, we will use the classification of ends.

\begin{proposition}\label{TypesCC}
	Let $I=(\eta_1,\eta_2)$ be a connected component of $\partial X\setminus \Lambda$ associated to a collared end $\varepsilon$ of $M$. Then,
	
	\begin{enumerate}
		\item if $\varepsilon$ is cylindrical, $Td(\eta_1 ,\eta_2)=\pi$ and there exists $\eta\in I$ such that $H_\eta(\infty)=[\eta_1,\eta_2]$,
		
		\item  if $\varepsilon$ is parabolic exceptional, $Td(\eta_1 ,\eta_2)\le\pi$ and there exists $\eta\in I$ such that $H_\eta(\infty)=[\eta_1,\eta_2]$,
		
		\item if $\varepsilon$ is expanding, $Td(\eta_1 ,\eta_2) =+\infty$. For $i=1,2$, the horocycle centered at $\eta_i$ satisfies $H_{\eta_i}(\infty)=\{\eta_i\}$. Moreover, if $\xi\in I$, then $H_\xi(\infty) \subset I$.
		
	\end{enumerate}
\end{proposition}

\begin{proof}
	First we observe that the closed interval $\bar I=[\eta_1,\eta_2]$ is the set at infinity of the lift $\tilde U$ of a neighborhood of $\varepsilon$ obtained from Theorem \ref{classification_ends}. If the end $\varepsilon$ is cylindrical, $\tilde U$ is a flat half-plane, so the Tits distance between $\eta_1 $ and $\eta_2$ is exactly $\pi$ and $\eta$ is the direction perpendicular to the geodesic bounding the flat half-plane. 
	If the end $\varepsilon$ is exceptional parabolic, there is no geodesic joining $\eta_1$ to $\eta_2$, so the Tits distance between them is less or equal than $\pi$ by \cite[Proposition 9.21]{BridsonHaefliger}. The point $\eta$ is the center of the horocycle that bounds the neighborhood $\tilde U$.
	 
	Now assume that $\varepsilon$ is expanding. Let $c$ be the geodesic that bounds $\tilde U$, which joins $\eta_1$ to $\eta_2$. We consider a geodesic ray $h$ starting at a point on $c$, perpendicular to it and pointing to the interior of $\tilde U$. Let $\gamma$ be the translation isometry of $c$. The curves $c$, $h$ and $\gamma h$ bound a fundamental domain of the expanding end. The isometry $\gamma$ does not fix the point at infinity of $h$, which we denote by $\xi$. Hence, the Tits distance between $\xi$ and $\gamma\xi$ is positive. The translates of $\xi$ by $\gamma$ form a sequence in $I$ of points which are at a positive constant Tits distance, so the distance between the two end points of $I$ is infinite. In fact, the same argument shows that the Tits distance between any point in $I$ and one of the extremities is infinite. As a consequence, for any $\eta \in I$, the set $H_\eta(\infty )= {\bar{B}_{Td}}(\xi,\pi/2)$ is contained in $I$. Also, it implies that the intervals $H_{\eta_i}(\infty)$ cannot intersect $I$. Since $\eta_i$ is not isolated in $\Lambda$, we have in fact $H_{\eta_i}(\infty)=\{\eta_i\}$.

\end{proof}

\begin{remark}
	Link, Peigné and Picaud observed that for exceptional parabolic ends obtained as surfaces of revolution, Clairaut's relation implies that $Td(\eta_1,\eta_2)=\pi$. It is not known if this property holds in general.
\end{remark}

We can now describe precisely the accumulations points of horocycles centered at a cylindrical or an exceptional parabolic end.

\begin{proposition} \label{horocycleparaboliccylindrical}
	Let $I=(\eta_1,\eta_2)$ be a connected component of $\partial X\setminus \Lambda $ associated to a cylindrical or an exceptional parabolic end, and let $\eta \in [\eta_1,\eta_2]$ be such that $H_\eta(\infty)=[\eta_1,\eta_2]$. We parametrize $[\eta_1,\eta_2] $ by a geodesic $c:[a,b]\to [\eta_1,\eta_2],\, -\pi/2\le a\le 0\le b\le \pi/2$, in the Tits distance such that $c(a)=\eta_1$, $c(0)=\eta$, $c(b)=\eta_2$. Then, for every $t\in[a,b]$,
	$$
	H_{c(t)}(\infty )= [c(\max(a,t-\pi/2)), c(\min(b,t+\pi/2))]
	$$
\end{proposition}

\begin{proof}
	We recall that two points in the boundary $\partial X$ within finite Tits distance can be joined by a geodesic of $\partial X$. The proposition follows then from the fact that intervals of horocycles at infinity are balls in the Tits distance of radius $\pi/2$.
\end{proof}

In general, there are some restrictions for two horocycles at infinity to intersect.

\begin{proposition}\label{orderpointsinfinity}
	Let $\eta ,\xi \in \partial X$ and denote $H_\eta(\infty)=[\eta_1,\eta_2]$ and $H_\xi(\infty)=[\xi_1,\xi_2]$. If $\xi_2\in [\eta_1,\eta_2)$, then $\xi\in [\xi_1, \eta)$.
\end{proposition}

\begin{proof}
	The distance between $\xi_1$ and $\eta_2$ is less or equal than $2\pi$. So there exists a geodesic in the Tits distance joining $\xi_1$ to $\eta_2$. The distance between $\eta_2 $ and $\xi_2$ must be strictly positive and the point $\xi_2$ is a distance exactly $\pi/2$ from $\xi$. This shows that the Tits distance between $\eta_2$ and $\xi$ is strictly bigger than $\pi/2$, while $Td(\eta_2,\eta)\le \pi/2$. So $\xi$ cannot be in the interval $[\eta, \xi_2]$.
\end{proof}

Two horocycles with distinct centers at infinity can accumulate at the same point at infinity. We next show that, even when this happens, the horocycles diverge.

\begin{proposition}\label{distance_horocycles}
	Let $v,w\in T^ 1X$ two vectors whose geodesics are not asymptotic. Then the horocycles of $v$ and $w$ diverge, that is, if we consider the parametrizations $\alpha,\beta:\reals\to X$ given by $\alpha(s):=\pi(h_s(v))$ and $\beta(s):=\pi(h_s(w))$, then
	$$d(\alpha(s),\beta(\reals_+))\xrightarrow[s\to+ \infty]{}+\infty,$$
	$$	d(\alpha(\reals_+),\beta(s))\xrightarrow[s\to+ \infty]{}+\infty.$$
\end{proposition}
\begin{proof}
	First we show that both formulas imply each other. Assume that 
	\begin{equation}\label{eq:div_1}
		\lim_{s\to +\infty} d(\alpha(\reals_+),\beta(s))=+\infty.
	\end{equation} 
	For every $s\in \reals_+$, let $t_s\in \reals_+$ such that $d(\alpha (s),\beta(\reals_+))=d(\alpha(s), \beta(t_s))$. We have 
	\begin{equation}\label{eq:div_2}
		d(\alpha(s),\beta(\reals_+))\ge d(\alpha(\reals_+),\beta(t_s)).
	\end{equation}
	Let $s_k$ be a sequence going to $+\infty$ such that
	$$
	\liminf_{s\to +\infty } d(\alpha(s),\beta(\reals_+)) =\lim_k d(\alpha(s_k),\beta(\reals_+)).
	$$
	Up to taking a further subsequence, we can assume that $t_{s_k}$ converges to a point $t$ in $\reals_+ \cup \{+ \infty\}$. If $t=+\infty$, then $d(\alpha (s),\beta(\reals_+))$ tends to infinity by (\ref{eq:div_1}) and (\ref{eq:div_2}). If $t\in \reals_+$, the sequence $\beta(t_{s_k})$ is bounded, so
	$$
	\lim_k d(\alpha(s_k),\beta(\reals_+))=\lim_k d(\alpha(s_k),\beta(t_{s_k}))=+\infty.
	$$
	
	We can also reduce the statement to the case where $\pi(v)=\pi(w)$. In the general case we can replace $w$ by $w':=g_{-\beta_{w_+}(\pi(v),\pi(w)) }w$, so that $v$ is based on the horocycle of $w'$. But the horocycle $\beta(\reals )$ of $w$ and that of $w'$ are equidistant, so $\alpha(s)$ diverges from one half-horocycle if and only if it diverges from the other.
	
	Let $\theta$ be the signed angle between $v$ and $w$. Up to interchanging the role of $v$ and $w$, we can assume that $\theta>0$. The function defined by $f(t):= \beta_{w_+}(c_v(t),\pi(w)) $ is convex. Moreover, $f(0)=0$ and $f'(0)=-\cos\theta$. 
	
	We distinguish two cases. Assume first that $f$ vanishes again for $t_0>0$, which means that the geodesic of $v$ escapes the horoball bounded by $\beta$. Therefore, there exists $s_0$ such that for $s\ge s_0$, $\beta(s)$ stays in the half-plane $X_v^ -$ on the left of $v$. But the half-horocycle $\alpha(\reals_+)$ lies in the half-plane $X_w^+$ on the right of $w$ and the distance from $\alpha (t)$ to $c_v$ goes to infinity.
	This shows that $d(\alpha(t), \beta([s_0,+\infty)))$ goes to infinity when $t\to +\infty$, which implies the other limit by the equivalence made at the beginning.

\begin{figure}
	\centering
	\includegraphics{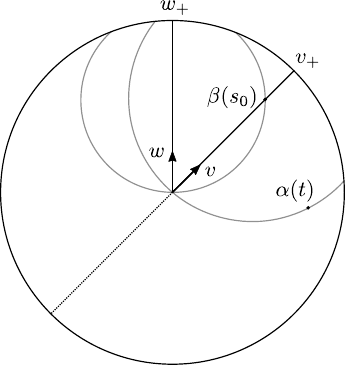}
	\caption{First case of the proof of Proposition \ref{distance_horocycles}.}
	\label{Fig:firstcase}
\end{figure}

	The second case is when $f$ does not vanish for positive time, which implies that $f$ is strictly negative. For $t>0$, the distance between $c_v(t)$ and $\beta(\reals)$ satisfies
	$$d(c_v(t),\beta(\reals))=|f(t)|\le \cos\theta \cdot t.$$
	For $s>0$, we consider the geodesic ray from $\beta(s)$ pointing to the end point of $w$, which intersects the geodesic ray of $v$ at a point $c_v(t_s)$. We have $d(c_v(t_s),\beta(s))\le \cos \theta \cdot t_s$. In particular, $t_s$ must be unbounded. Now,
	\[
	t_s\le d(\alpha(\reals_+),c_v(t_s))\le d(\alpha (\reals_+),\beta(s))+d(\beta(s),c_v(t_s))\le d(\alpha (\reals_+),\beta(s))+\cos \theta \cdot t_s.
	\]
	Passing to the limit, we see $\lim_{s\to +\infty}d(\alpha (\reals_+),\beta(s))=+\infty$. We apply again the equivalence to obtain the other limit. 
	
	\begin{figure}
		\centering
		\includegraphics{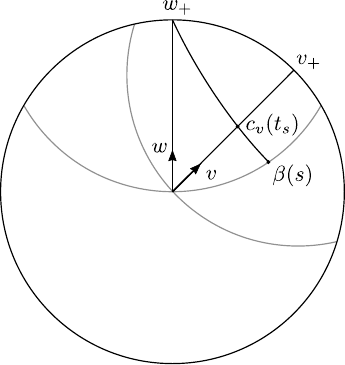}
		\caption{Second case of the proof of Proposition \ref{distance_horocycles}.}
		\label{Fig:secondcase}
	\end{figure}

\end{proof}

\subsection{Geometrically finite surfaces}

We are concerned with geometrically finite surfaces $M$. Recall that this means that the fundamental group of $M$ is finitely generated. The topology and the geometry of the ends of such surfaces is easier to understand than for arbitrary surfaces.

\begin{theorem} \cite{Kerekjarto}
	Let $M$ be a (topological, connected, complete) geometrically finite surface. Then $M$ has a finite number of ends and each of them is collared.
\end{theorem}

In fact $M$ can be obtained by removing a finite number of points of a compact surface. We will now describe a way of decomposing our nonpositively curved surface $M$ into a compact part and some neighborhoods of the finitely many ends. We start by choosing suitable neighborhoods of the ends.

Since $M$ is geometrically finite, the number of ends is finite and each one is collared.
Propositions \ref{parabolic_ngh} and \ref{funnel_ngh} provide an almost canonical way of choosing standard neighborhoods $E_1,\dots ,E_r$, $r\ge 1$, of the ends of $M$. 
Up to shrinking sufficiently enough the neighborhoods of the parabolic ends, which consist of horoballs, we can assume that $E_1,\dots ,E_r$ are pairwise disjoint. 
The remaining part of $M$,

$$K:= M\setminus\bigcup_{i=1}^r E_i,$$
will be called the \emph{compact part} of $M$.

\begin{proposition}\label{Geometrically_finite_limit_points}
	Let $M$ be a nonpositively curved surface with finitely generated fundamental group. Then the limit set $\Lambda$ only contains radial limit points and finitely many $\Gamma$-orbits of parabolic limit points.
\end{proposition}

\begin{proof}
	We consider the map which associates a parabolic limit point to the corresponding end  the manifold $M$. Each fiber of this map consists of one or two $\Gamma$-orbits (one if the end is simple parabolic and two if the end is exceptional parabolic). Since the number of ends is finite, the parabolic limit points are in a finite union of $\Gamma$-orbits.
	
	Let $\xi\in \Lambda$ and let $c$ be a geodesic ray on $X$ such that $c(+\infty )=\xi$.  We consider the lift to $X$ of the standard neighborhoods of $M$,
	$$
	Y:=\Pi^ {-1}(E_1\cup \dots\cup E_r).
	$$
	There are two possibilities. If $c$ leaves every connected component of the set $Y$, then there exists a sequence $t_n\to \infty$ such that $c(t_n)\in \tilde K:=\Pi^ {-1}(K)$. Let $o\in \tilde K$. There exist isometries $\gamma_n$ in $\Gamma$ such that $d(c(t_n),\gamma_n o)\le \diam (K)<+\infty$. This means that $\xi$ is radial by definition.
	
	The other possibility is that $c$, after some time, stays in a connected component of $Y$, say $Y_0$. This component $Y_0$ is a lift of a standard neighborhood $E_i$ of and end of $M$. If the end is parabolic, then $Y_0$ is a horoball and there is a parabolic isometry $\gamma$ associated to the end which fixes $Y_0$ and such that $E_i=Y_0/\langle \gamma \rangle $. But $\xi\in  \Lambda \cap \overline{Y_0}$, so it can only be a simple parabolic or an exceptional parabolic point. If the end is cylindrical or expanding, then $Y_0$ is a half-plane and there exists an axial isometry $\gamma$ leaving invariant $Y_0$. Again, $\xi\in  \Lambda \cap \overline{Y_0}$, so it is the end point of an axis of $\gamma$, which implies that it is radial.

\end{proof}

The complement of the limit set $\Lambda$ is formed by a countable union of intervals. Each of these intervals corresponds to an end of the surface, which can be either expanding, cylindrical or exceptional parabolic. 
\section{Proofs of the results}

Recall that we chose an orientation of $X$ so that the induced orientation on $\partial X$ is counterclockwise. This gives the horocycles of $X$ a natural orientation which allowed us to define the horocyclic flow $h_s$ on $T^ 1X$. We assume that $M$ is orientable, so every element of $\Gamma$ is orientation-preserving and the orientation of $X$ descends to an orientation of $M$. Hence, the horocyclic flow $h_s$ is defined on $T^1M$ unambiguously. 
A positive half-horocycle on $T^1M$ is a positive orbit $h_{\reals_+}(v)$ of the horocyclic flow.

\subsection{Horocyclic limit points and dense orbits}

We will first establish a relation between vectors whose full horocylic orbit is dense in a certain set and horocyclic limit points. The result we prove will also follow from the study of half-horocycles later, but here we present a direct proof following \cite{Dalbo00}. 

We fix a reference point $o\in X$. A horocycle $H$ is uniquely determined by its point at infinity $\xi\in \partial X$ and the Busemann cocycle $t=\beta_\xi(H,o)\in \reals $. We use the data $(\xi,t)\in \partial X\times \reals$ as coordinates on the set of horocycles of $X$.   
The natural action of $\Gamma $ on the set of horocycles is given in these coordinates by 
$$\gamma (\xi, t)=(\gamma \xi, t+\beta_\xi(o,\gamma^ {-1}o)), \quad \forall \gamma\in\Gamma, \, (\xi,t)\in \partial X\times \reals.$$

We denote by $\Omega^+$  the set of vectors $v$ in $T^1M$ such that, for any lift $\tilde v$ to $T^ 1 X$, the positive end point $v_+$ is in the limit set $\Lambda$. The set $\Omega ^+$ is invariant by the horocyclic flow. 
The set of horocycles contained in $\Omega^ +$ is identified with the set $\Lambda\times \reals $ in the coordinates defined above. 

The following result completely characterizes dense horocycles in $\Omega^ +$.

\begin{proposition}\label{dense_horocyclic} Let $M$ be a geometrically finite nonpositively curved surface. 
	Let $v\in \Omega^ +$. The stable horocycle of $v$ is dense in $\Omega^+$ if and only if for any lift $\tilde v$ of $v$ the end point $\tilde v_+$ is horocyclic.
\end{proposition}

In the proof of this proposition, we will use the following standard fact.

\begin{lemma}\label{duality_density}
	Let $v\in T^ 1M $. The stable horocycle of $v$ is dense in $\Omega^+$ if and only if for any lift $\tilde v$ of $v$ the $\Gamma$-orbit of horocycle determined by $\tilde v$ is dense in $\Lambda\times \reals$.
\end{lemma}

	The proof of the lemma is based on the fact that the action of the horocyclic flow on $T^1 M$ is dual to the action of $\Gamma$ on the set of horocycles on $X$. More precisely, given $v, w\in T^1M$ with lifts $\tilde v, \tilde w\in T^1 X$, there exists a sequence $s_n$ in $ \reals$ such that $h_{s_n}v$ converges to $w$ if and only if there exists a sequence $\gamma_n$ in $\Gamma$ such that the images of the horocycle of $\tilde v$ by $\gamma_n$ converge to the horocycle of $\tilde w$. Now we can prove Proposition \ref{dense_horocyclic}.

\begin{myproof}{Proposition}{\ref{dense_horocyclic}}
	We begin by proving that $\tilde v_+$ is horocyclic, provided that the stable horocycle of $v$ is dense in $\Omega^+$.
	In view of Lemma \ref{duality_density}, we know that the $\Gamma $-orbit of the horocycle $h_{\reals} \tilde v$ is dense in the set of horocycles with coordinates in  $\Lambda\times \reals$. For $n\ge 1$, the horocycle $h_{\reals }g_n \tilde v$ has coordinates $(\tilde v_+, \beta_{\tilde v_+}(\pi(\tilde v),o)+n)$.

	 By the density of $\Gamma h_{\reals}\tilde v$, we can find $\gamma_n\in \Gamma$ such that $\gamma_n\tilde v_+ \to \tilde v_+$ and 
		$$|\beta_{\tilde v_+}(\pi(\tilde v),o)+n-\beta_{\gamma_n \tilde v_+}(\gamma_n \pi (\tilde v),o)|=|n -\beta_{\tilde v_+}(o,\gamma_n^{-1}o)|\le 1.$$
	This shows that $\lim_n \beta_{\tilde v_+}(\gamma_n^{-1}o,o)=-\infty$.
	
	Now we extract a subsequence $\gamma'_n$ of $\gamma_n^ {-1}$ such that $\gamma'_n o$ converges to a point $\xi\in \Lambda$. If we can show that $\xi=\tilde v_+$, we obtain that $\tilde v_+$ is horocyclic as we wanted. The point $\xi$ must be in $H_{\tilde v_+}(\infty)$. Since $\tilde v_+\in H_{\tilde v_+}(\infty)$ lies in $\Lambda$, the set $H_{\tilde v_+}(\infty)$ must be of the form $\{\tilde{v}_+\}$, $[\tilde{v}_+, \eta]$ or $[\eta, \tilde{v}_+]$ for some $\eta \in \partial X\setminus \{\tilde{v}^+\}$. In the first case, we have $\xi=\tilde v_+$ as desired. In the second case, the interval $(\tilde{v}_+, \eta)$ is in a connected component of $\partial X\setminus\Lambda$. This component is associated to an end of the manifold.
	
	We can exclude expanding ends because the Tits distance between the extremity $\tilde v_+$ of $H_{\tilde v_+}(\infty)$ and a point in the interior of $H_{\tilde v_+}(\infty)$ is infinity, but here $Td(\tilde{v}_+, \eta)\le \pi/2$. So it must be cylindrical or exceptional parabolic. In the first case, $\xi \in  \Lambda \cap [\eta,\tilde v_+]=\{\tilde v_+\}$, so we are done. Let us deal with the case that the end is exceptional parabolic. Recall that $\Lambda $ does not intersect the the interior of $H_{\tilde v_+}(\infty)$, so the only way for $\xi$ not being equal to $\tilde v_+$ is that $\xi=\eta \in \Lambda$.
	
	We assume that the latter happens and we show that  $\tilde v_+$ is horocyclic anyway (there is a sequence different from $\gamma_n'$ satisfying the definition). Let $\gamma $ be the isometry associated to the exceptional parabolic end and let $\bar \xi$ be the point at infinity whose horocycles are left invariant by $\gamma$. Necessarily $\bar \xi$ is different from $\tilde v_+$ because vectors in $T^1 M$ whose lift points to $\bar\xi$ have periodic horocycles and here the horocycle of $v$ is dense by hypothesis. Moreover $H_{\bar \xi} (\infty)=[\tilde v_+,\eta]$. Up to taking the inverses, $\gamma^ n \pi(\tilde v)$ belongs to $H_{\bar\xi}(\pi(\tilde v))$ and converges to $\tilde v_+$. Let $\tilde w $ be the vector with base point $\pi(\tilde v)$ that points to $\bar \xi$. The half-horocycle $h_{\reals_+} \tilde w$ stays in the horoball bounded by $h_{\reals}\tilde v$ and, by Proposition \ref{distance_horocycles}, they diverge one from the other. This implies that $\beta_{\tilde v_+ }(\gamma^n o,o)$ tends to $-\infty$, showing that $\tilde v_+$ is horocyclic.

	For $H_{\tilde v_+}(\infty)=[\eta, \tilde{v}_+]$ the proof is analogous.
	
	We now assume that $\tilde v_+$ is horocyclic and show that the horocycle $h_{\reals}v$ is dense in $\Omega^+$, or equivalently, that the $\Gamma$-orbit of the horocycle determined by $\tilde v$ is dense in the space $\Lambda\times \reals$. This is already known when $v$ is periodic for the geodesic flow (see the proof of Theorem A $i)$ $\Rightarrow $ $ii)$ in \cite{LinkPeignePicaud}).
	
	For the general case, it is enough to show that if $\tilde v_+$ is horocyclic the closure of the $\Gamma$-orbit of $h_{\reals}\tilde v$ contains a horocycle with a periodic vector. Assume that $\tilde v_+$ is horocyclic: there exists a sequence $\gamma_n$ such that $\gamma_n o$ tends to $\tilde v_+$ and $\beta_{\tilde v_+}(\gamma_n o,o)$ tends to $-\infty$. Up to passing to a subsequence we can assume that $\gamma_n^ {-1}\tilde v_+$ converges to $\eta\in \partial X$. We choose an axial isometry $\gamma$ whose axis does not bound a flat half-plane such that $\eta\not \in \{\gamma^\pm \}$. This is possible because the pairs of end points of isometries not bounding a flat half-plane are dense in $\Lambda\times \Lambda$. Up to taking another subsequence of $\gamma_n$, we can also choose a sequence $r_n$ of integers going to $+\infty$ such that
	$$\beta_{\tilde v_+}(\gamma_n o ,o)+r_n l(\gamma)$$ converges to some $\lambda\in \reals$.
	
	We claim that the images under $\gamma ^ {r_n}\gamma_n^ {-1} $ of the horocycle $h_{\reals}\tilde v$ tend to a horocycle with coordinates $(\gamma^+, \alpha )$, $\alpha\in \reals$. We first observe that, since $\gamma$ is axial and its axis does not bound a flat half-plane, compact subsets of $\partial X\setminus\{\gamma^ -\}$ converge uniformly to $\gamma^ +$ \cite[Lemma III.3.3]{ballmannlectures}. So $\gamma^ {r_n}\gamma_n^ {-1}\tilde v_+$ tends to $\gamma^+$. The isometries $\gamma ^ {r_n}\gamma_n^ {-1}$ act on the second coordinate of $\partial X\times \reals$ by adding the quantity $\beta_{\tilde v_+}(o,(\gamma ^ {r_n}\gamma_n^ {-1})^ {-1}o)$. We write
	
	$$ \beta_{ \tilde v_+}(o,(\gamma ^ {r_n}\gamma_n^ {-1})^ {-1}o)=\beta_{\tilde v_+}(o,\gamma_no)+ \beta_{ \tilde v_+}(\gamma _n o,\gamma_n \gamma ^ {-r_n}o),
	$$ where the second term is equal to
	$$
	\beta_{ \tilde v_+}(\gamma _n o,\gamma_n \gamma ^ {-r_n}o) =\beta_{ \gamma_n^ {-1} \tilde v_+}(o, \gamma ^ {-r_n}o)= 2\langle \gamma_n^ {-1}\tilde v_+ , \gamma ^ -\rangle_o -r_n l(\gamma) -2\langle \gamma^ {r_n} \gamma_n^ {-1}\tilde v_+ , \gamma ^ - \rangle_o.
	$$
	The Gromov product taking values in $\reals_+\cup \{\infty\}$ is continuous on $\Lambda\times \Lambda$. Taking the limit in the previous equations, we obtain that $\beta_{ \tilde v_+}(o,(\gamma ^ {r_n}\gamma_n^ {-1})^ {-1}o)$ converges to 
	$$-\lambda +2\langle \eta , \gamma ^ - \rangle_o -2\langle \gamma^ + , \gamma ^ - \rangle_o.
	$$
	Since $\gamma^-$ is the enpoint of an axis not bounding a flat half-plane, there is a geodesic joining it to any other point at infinity, which implies that the Gromov products above are finite.
\end{myproof}

\subsection{Half-horocycles}

This section is devoted to understanding when the half-horocycles are dense in $\Omega^+$. It is divided in two parts: half-horocycles of vectors in $\Omega^+$ and half-horocycles of vectors outside of $\Omega^+$. The first part is inspired by the article \cite{Schapira11}, which does the same for hyperbolic surfaces. The pinched negative curvature case can be treated in the same way, but nonpositively curved surfaces are much more subtle.

\subsubsection{Dense half-horocycles in $\Omega^ +$}

Given a vector $\tilde w\in T^1 X$, the geodesic $c_{\tilde w}$ generated by $\tilde w$ divides $X$ into two half-planes.  One of the half-planes is oriented such that the orientation induced in the boundary $c_{\tilde w}$ is given by $\tilde w$. We denote this half-plane by $ X_{\tilde w}^-$ and the other one by $ X_{\tilde w}^+$.

G. Knieper proved that horocycles with rank $1$ $g_t$-recurrent vectors are contracting and expanding. The following lemma explains what happens in general for recurrent vectors.

\begin{lemma}\label{lemma_Knieper} \cite[Proposition 4.1]{Knieper98}
	Let $v$ be a positively (resp. negatively) recurrent vector in $T^1M$ such that, for any lift $\tilde v$ of $v$ to $T^1 X$, the geodesic $c_{\tilde v}$ does not bound a flat strip in $X_{\tilde v}^+$. Then, the positive stable (resp. unstable) half-horocycle of $\tilde v$ is contracting for positive (resp. negative) time, that is, for every $s\ge 0$,
	 $$
	 \lim_{t\to + \infty} d(g_t(\tilde v), g_t(h_s(\tilde v)))=0.
	 $$
	
	In particular, if $v$ has rank $1$  and is positively (resp. negatively) recurrent, then its stable (resp. unstable) horocycle coincides with its stable (resp. unstable) manifold and it consists of rank $1$ vectors exclusively.
\end{lemma}

We will also need the following refinement of \cite[Lemma III.3.3]{ballmannlectures}.

\begin{lemma}
	Let $\gamma$ be an axial isometry and let $\tilde v$ be a vector tangent to an axis of $\gamma$. If the half-plane $X_{\tilde v}^ +$ is not flat, then $\gamma $ does not fix any point in $(\tilde v -,\tilde v_+)$. In fact, for every $x\in(\tilde v -,\tilde v_+)$, $\lim_{n\to+\infty} \gamma^{\pm n}\xi =\tilde v_{\pm}$ and the convergence is uniform on compact subsets of $(\tilde v -,\tilde v_+)$. 
\end{lemma}

We denote the set of vectors $v$ in $T^1M$ which have a lift in $T^1 X$ with both end points in the limit set $\Lambda$ by $\Omega$. The set $\Omega$ is precisely the set of $g_t$-nonwandering vectors \cite[Section III.1]{ballmannlectures}. Among these vectors, we consider the subset $\Omega_{NF}\subset \Omega$ of those whose geodesic does not bound a flat strip.
	 
\begin{lemma}\label{lemma_Schapira}
	$\Omega^+ =\overline{ h_{\reals_+} (\Omega_{NF})}$
\end{lemma}

\begin{proof}
	Let $v\in \Omega^+$ and take a lift $\tilde v$ in $T^1 X$. By definition, the positive end point $\tilde v_+$ is in the limit set $\Lambda$. 
	First, assume that the limit set $\Lambda$ intersects the interval $(\tilde v_+, \tilde v_-)$. Since the action of $\Gamma$ is minimal on the limit set $\Lambda$, there exists $\xi\in \Lambda \cap (\tilde v_+,\tilde v_-)$ which is an end point of a lift of a  rank $1$ periodic geodesic. The points $\xi$ and $\tilde v_+$ can be joined by a unique geodesic, so we can choose $\tilde w\in h_{\reals} \tilde v$ such that $\tilde w_{-}=\xi$. Notice that the projection $w$ of $\tilde w$ to $T^1 M$ is in the set $\Omega_{NF}$. By the choice of $\xi$, it is clear that there exists $s\ge 0$ such that $h_s(w)= v$.

	If  $(\tilde v_+, \tilde v_-)$ does not intersect $\Lambda$, then $v_+$ is the first end point of an interval of $\partial X\setminus\Lambda$ of the form $I=(\tilde v_+, \eta)$ with $\eta \in [\tilde v_-, \tilde v_+)$. The end corresponding to $I$ is either cylindrical or expanding.	
	Since $\tilde v_+$ is not isolated in $\Lambda$ we can choose a sequence $\eta_n$ in $\Lambda\setminus\{\tilde v_+\}$ of end points of lifts of closed rank $1$ geodesics converging to $\tilde v_+$. Necessarily, $\eta_n \in [\tilde v_-,\tilde v_+)$. Let $\tilde v_n$ be the vector with the same base point as $\tilde v$ that points at $\eta_n$. As $n$ goes to infinity, $\tilde v_n$ tends to $\tilde v$. Since $\tilde v_+$ can be joined to $\eta_n$ for all $n$, we can consider vectors $\tilde w_n\in h_{\reals} \tilde v_n$ such that $\tilde w_{n-}=\tilde v_+$. Since $\tilde v_{n-}$ tends to $\tilde v_-$, for large $n$, $\tilde v_{n-}\in (\tilde v_+, \tilde v_-)$.  This implies that the vector $\tilde w_n$ lies in the negative horocycle orbit of $\tilde v_n$, so there exists $s_n\ge0$ such that $\tilde v_n = h_{s_n}(\tilde w_n)$. The projections $w_n$ of $\tilde w_n$ to $T^ 1M$ belong to $\Omega_{NF}$ and satisfy $\lim_n h_{s_n}(w_n)=v$. 
	
	\begin{figure}
		\centering
		\includegraphics{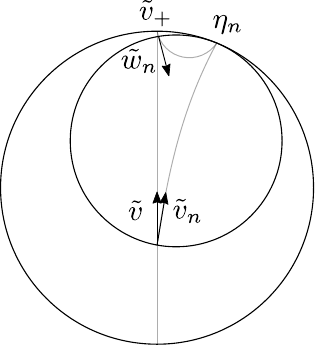}
		\caption{Elements of the proof of Lemma \ref{lemma_Schapira}.}
		\label{fig:lemma35}
	\end{figure}
	
\end{proof}

\begin{lemma}\label{transitivity}
	There exists a vector $w$ in $\Omega$ whose negative orbit $g_{\reals_-}(w)$ under the geodesic flow is dense in the set $\Omega_{NF}$ of nonwandering vectors which do not bound a flat strip.
\end{lemma}

\begin{proof}
	Coudène and Schapira showed that the geodesic flow in restriction to the closure of the set of periodic hyperbolic vectors is positively (hence negatively) transitive \cite{CoudeneSchapira10,CoudeneSchapira14}. This implies the existence of a dense negative orbit $g_{\reals_-}(w)$ on this set. On the other hand, Link, Peigné and Picaud proved that the set of pairs of end points of rank $1$ periodic vectors is dense in $\Lambda\times \Lambda$ \cite[Corollaire 1.20]{LinkPeignePicaud}. This implies that the set $\Omega_{NF}$ is contained in the closure of the set of rank $1$ periodic vectors as we show next.
	
	Let $v$ be a nonwandering vector in $\Omega_{NF}$ and take a lift $\tilde v\in T^ 1X$. We know that $\tilde v_-$ and $\tilde v_+$ are in the limit set $\Lambda$. There exists a sequence of geodesics $c_n$ on $X$ whose projection to $M$ is closed and such that $c_n(+\infty)$ converges to $\tilde v_+$ and $c_n(-\infty)$ converges to $\tilde v_-$. We can parametrize $c_n$ such that $c_n(0)$ is the point on $c_n$ closest to the base point of $v$. We apply \cite[Lemma III.3.1]{ballmannlectures}: for $R>0$, if $n$ is big enough, any geodesic joining $c_n(-\infty)$ to $c_n(+\infty)$ is at distance less than $R$ from $c(0)$. The vectors $\dot{c}_n(0)$ converge to $v$.
\end{proof}

\begin{remark}
	If $v\in\Omega$ does not bound a flat half-plane, we can still apply the previous argument by choosing $R>0$ larger than the width of the maximal strip containing the geodesic generated by $v$. In this case, the argument shows that there is a vector generating a geodesic biasymptotic to $v$ which is in the closure of the set of hyperbolic periodic geodesics.
	
	If $v\in \Omega$ and its geodesic bounds a maximal flat half-plane, then it is also in the closure of the set of hyperbolic periodic geodesics. In fact, the geodesic of $v$ bounds a cylindrical end. The periodic geodesics in the interior of the cylinder are not in the closure of hyperbolic periodic vectors. This is already studied in \cite{CoudeneSchapira2011}.
\end{remark}

Recall that the boundary at infinity $\partial X$, which is homeomorphic to a circle, is oriented in the counterclockwise sense. An interval $I$ in $\partial X$ has two end points, which we call \emph{first end point} and \emph{second end point} of $I$ according to the orientation of $\partial X$.

\begin{proposition}\label{characterization_dense_hor_periodic}
	Let $v$ be a periodic vector in $T^1 M$. 
	Then, the half-horocycle $h_{\reals_+}(v)$ is dense in $\Omega^ +$ if and only if, for any lift $\tilde v$ of $v$ to $T^ 1X$, $\tilde v_+$ is not the second end point of an interval in $\partial X\setminus \Lambda $.
\end{proposition}

\begin{proof}
	First, assume that $\tilde v_+$ is the second end point of an open interval $I$ of $\partial X\setminus \Lambda $. Then the interval $I$ is bounded by the points $\tilde v_+$ and $\tilde v_-$. So the end associated to $I$ is bounded by a closed geodesic. For $s\ge 0$, $h_s(v)$ is based on the end bounded by the closed geodesic generated by $v$. Moreover, the distance between $h_s(v)$ and this geodesic tends to $+\infty$ when $s$ goes to $+\infty$. So the half-horocycle $h_{\reals_+}(v)$ is divergent.
	
	For the converse, we will adapt \cite[Proposition 3.12]{Schapira11}. Assume that $\tilde v_+$ is not the second end point of an interval in $\partial X\setminus \Lambda $. Notice that the half-plane $X_{\tilde v}^+$ is not flat. It is enough to prove the result with the additional assumption that the geodesic generated by $\tilde v$ does not bound a flat strip in $X^+_{\tilde v}$. 
	
	\textit{Step 1.} We prove that $\Omega_{NF}\subset \overline{g_{\reals} h_{\reals_+}(v)}$.
	By Lemma \ref{transitivity}, there exists a vector $w$ in $\Omega_{NF}$ whose negative orbit $g_{\reals_-}(w)$ is dense in $\Omega_{NF}$. Since $\tilde v_+$ is not the second end point of an interval in $\partial X\setminus \Lambda$, we can lift $w$ to a vector $\tilde w$ in $T^ 1X$ with $\tilde w_-$ in the interval $(\tilde v_-,\tilde v_+)$. Moreover, there is a geodesic joining $\tilde w_-$ to $\tilde v_+$. We take the vector $\tilde u\in h_{\reals} \tilde v$ tangent to this geodesic, and we project it to a vector $u$ in $T^1 M$. By the choice of the point $\tilde u_-=\tilde w_-$, $u$ is in the positive horocyclic orbit $h_{\reals_+}(v)$ of $v$. By Lemma \ref{lemma_Knieper} the unstable horocycle of $\tilde w$ is contracting in the past, so the negative geodesic orbit of $u$ is also dense in $\Omega_{NF}$. Hence, $\Omega_{NF}\subset \overline{g_{\reals}(u)}\subset \overline{g_{\reals} h_{\reals_+}(v)}$.

\textbf{Claim.} Let $\varepsilon>0$. There exists a $g_t$-periodic vector $u$ in $T^1 M$ whose geodesic does not bound a flat half-plane and such that there exist $n,m\in \naturals$ with $|ml(v)-nl(u)|<\varepsilon$. 
\begin{proof}
	We first follow the proof of the nonarithmeticity of the length spectrum for negatively curved surfaces by Dal'bo \cite{Dalbo99}. There exist two isometries $\gamma_1$ and $\gamma_2$ with distinct axis $c_1$ and $c_2$ that intersect. Neither of the axis bounds a flat plane, so each has a repelling fixed point $\gamma_i^-$ and an attracting fixed point $\gamma_i^+$ at infinity. Up to taking the inverse of $\gamma_1$ we can assume that they are ordered in counterclockwise sense as follows: $\gamma_1^ +,\gamma_2^-,\gamma_1^-,\gamma_2^ +$ (see Figure \ref{fig:claim}). Consider the isometry $g=\gamma_1 \gamma_2$. We observe that $g [\gamma_2^+,\gamma_2^ -] \subset (\gamma_2^+,\gamma_2^ -)$ and  $g^ {-1} [\gamma_2^-,\gamma_2^ +] \subset (\gamma_2^-,\gamma_2^ +)$. 
	The only type of isometry which achieves this is an axial isometry whose axis does not bound a flat half-plane. Moreover, the axis of $g$ intersects that of $\gamma_2$ because of the position of the fixed points.
	
		\begin{figure}
		\centering
		\includegraphics{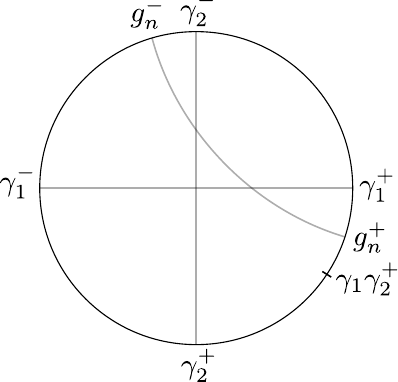}
		\caption{Order of the fixed points of $\gamma_1 $ and $\gamma_2$ at infinity.}
		\label{fig:claim}
	\end{figure}
	
	Similarly, the isometry $g_n=\gamma_1\gamma_2^ n$, $n\ge 0$, is axial, its axis does not bound a flat half-plane and it intersects the axis of $\gamma_2$.
	We observe that the attracting fixed point $g_n^+$ of $g_n$ tends to $\gamma_1\gamma_2^ +$ and the repelling fixed point $g_n^ -$ of $g_n$ tends to $\gamma_2^ -$ when $n$ goes to infinity. We have
	\begin{align*}
		l(g_n)-l(g_{n-1})&= \beta_{g_n^ -}(o, \gamma _2 ^ {-n} \gamma_1^ {-1} o) +
		\beta_{g_{n-1}^ +}(\gamma_1 \gamma_2^ {n-1} o,o) 
		\\
		&=\beta_{g_n^ -}(o, \gamma_2^ {-1} o) + \beta_{g_n^ -}(\gamma_2^ {-1} o, \gamma _2 ^ {-n} \gamma_1^ {-1} o) +
		\beta_{g_{n-1}^ +}(\gamma_1 \gamma_2^ {n-1} o,o)
		\\
		&=\beta_{g_n^ -}(o, \gamma_2^ {-1} o) + \beta_{g_n^ -}(\gamma_1 \gamma_2^ {n-1} o,  o) +
		\beta_{g_{n-1}^ +}(\gamma_1 \gamma_2^ {n-1} o,o) 
		\\
		&=  \beta_{g_n^ -}(o, \gamma_2^ {-1} o) + 2 \langle g_n^ - , g_{n-1}^ + \rangle_{\gamma_1 \gamma_2^ {n-1} o} - 2 \langle g_n^ - , g_{n-1}^ + \rangle_o
		\\
		&=  \beta_{g_n^ -}(o, \gamma_2^ {-1} o) + 2 \langle \gamma_2 g_n^ - , g_{n-1}^ + \rangle_{ o} - 2 \langle g_n^ - , g_{n-1}^ + \rangle_o,
	\end{align*}
	which converges to $\beta_{\gamma_2^ -}(o, \gamma_2^ {-1} o)=l(\gamma_2)$ when $n $ goes to infinity.
	
	If $z_{n-1}$ is a point in the intersection of $c_2$ with the axis of $g_{n-1}$, we have
	\begin{align*}
		l(g_n)&\le d(g_n z_{n-1}, z_{n-1}) 
		\\
		&< d(\gamma_{1} \gamma_2^ {n} z_{n-1}, \gamma_{1} \gamma_2^ {n-1} z_{n-1}) + d(\gamma_{1} \gamma_2^ {n-1} z_{n-1}, z_{n-1}) = l(\gamma_2) + l(g_{n-1}),
	\end{align*}
	where the last inequality is strict because the axis of $g_{n-1} $ is distinct from that of $\gamma_2$.
	
	To conclude the proof of the claim, we use the following arithmetic lemma. The proof of the lemma is postponed until the end of the proof of Proposition \ref{characterization_dense_hor_periodic}.
	
	\begin{lemma}\label{lemma_nonarith}
		Let $\alpha >0$ and $(\beta_k) $ be a sequence of real numbers such that, for all $k$, $\beta_k-\beta_{k-1}<\alpha$ and $\lim_k \beta_k-\beta_{k-1}=\alpha$. Then, for all $\varepsilon>0$, there exist $k,n,m\in \naturals$ such that
		$$
		0<|m\alpha -n\beta_k| <\varepsilon.
		$$
	\end{lemma}
	
	
	We have seen that the numbers $\alpha= l(\gamma_2)$ and $\beta_k=l(g_k)$ satisfy the hypothesis of the lemma. We deduce that, for every $\varepsilon'>0$, there exist $k,n',m'\in \naturals$ such that 
	$$
	0<|m'l(\gamma_2)- n'l(g_k)|<\varepsilon'.
	$$
	
	Finally, we compare $l(v)$ with $l(\gamma_2)$. If the ratio $l(v)/l(\gamma_2)$ is irrational, we can find $n,m\in \naturals$ such that 
	$$
	0<|ml(v)- nl(\gamma_2)|<\varepsilon
	$$
	by Dirichlet's approximation theorem. If it is equal to a rational number $p/q$, $p,q\in \naturals^ *$, then putting $\varepsilon'=\varepsilon/p$, $m=m'q$ and $n=n'$, we obtain
	$$
	0<|ml(v)- nl(g_k)|<\varepsilon.
	$$
	The vector $u$ in the claim is tangent to an axis of $\gamma_2$ or $g_{n}$.
\end{proof}

	\textit{Step 2.} We prove that $g_{\reals} (v)\subset \overline{h_{\reals_+}(v)}$. 
	
	Let $\varepsilon>0$. We write $\varepsilon_0:= ml(v)-nl(u)$, where $u$, $n$ and $m$ are given by the claim. Since $ v_+$ is not the second end point of an interval in $\partial X\setminus\Lambda$, it is possible to take a lift $\tilde u$ of $u$ to $T^ 1 X$ such that $\tilde u_{-}$ is between $\tilde v_-$ and $\tilde v_+ $. We denote by $\gamma$ and $\gamma_0$ the isometries that translate the vectors $\tilde v$ and $\tilde u$ respectively.

Since $\tilde v _+$ is  not the second end point of an interval in $\partial X\setminus \Lambda $, for any point in $(\tilde v_-,\tilde v_+)$ there is a geodesic joining it to $\tilde v_+$. For $j\in \naturals$ big enough, $\gamma_0^ {-j}\tilde v_-$ is close to $\tilde u_{-}$. So we can consider a vector $\tilde v_j$ which points positively to $\tilde v_+$, negatively to $\gamma_0^ {-j} \tilde v_{-}$ and $\tilde v_j\in h_{\reals} \tilde v$. In fact $\tilde v_j$ is in the positive half-horocycle $h_{\reals_+}\tilde v$ of $\tilde v$. We also consider the vector $\tilde w\in h_{\reals} \tilde v$ with $\tilde w_-=\tilde u_-$, which is the limit of the vectors $\tilde v_j$ when $j$ goes to $+\infty$.

	\begin{figure}
		\begin{center}
			\includegraphics[scale=1.4]{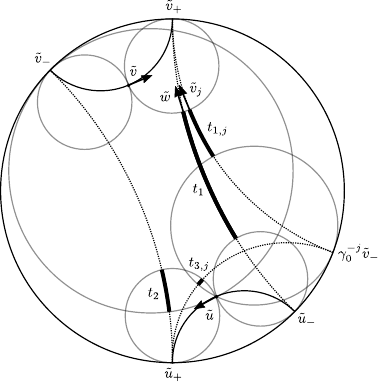}
		\end{center}
	\caption{Elements of the proof of Step 2.}
	\label{fig:step2}
\end{figure}

For $i\in \naturals$, we define a new vector $\tilde u_{i,j}= \gamma ^i g_{-il(v)} \tilde v_j $ obtained by pushing $\tilde v_j$ backwards by the geodesic flow a time $i l(v)$ and then bringing it back to $h_{\reals}\tilde v$ using the isometry $\gamma$. The vectors $\tilde u_{i,j}$ are again in the positive half-horocycle $h_{\reals_+}(\tilde v)$ of $\tilde v$. We observe that, for $j\in \naturals$ large enough and for all $i\in \naturals$, the vector $\gamma_0^j \gamma^{-i} \tilde u_{i,j}$ is in the weak unstable manifold of $\tilde v$. Wisely choosing $i$ and $j$, we will be able to control which horocycle this vector belongs to.

We set the following notations:
\begin{align*}
	t_0:=&\beta_{\tilde v_-}(\pi(\tilde v), \pi(\tilde u)),
	\\
	t_{1,j}:=&\beta_{  \gamma_0^{-j} \tilde v_-}(\pi(\tilde u), \pi(\tilde v_{j})),
	\\
	t_1:=&\beta_{ \tilde u_-}(\pi(\tilde u), \pi(\tilde w)),
	\\
	t_2:=&2 \langle \tilde v_- , \tilde u_+ \rangle_{\pi(\tilde u )},
	\\
	t_{3,j}:=&- 2\langle \gamma_0^ {-j } \tilde v_- , \tilde u_+ \rangle_{ \pi(\tilde u)}.
\end{align*} 
By continuity, $\lim_j t_{1,j} =t_1$ and $\lim_j t_{3,j}=- 2\langle \tilde u_- , \tilde u_+ \rangle_{ \pi(\tilde u)}=0$. 
The quantity $t_2$ is finite, because there is a geodesic joining $\tilde v_-$ to $\tilde u_+$. Finally, we let $T:=t_0+t_1+t_2$ and we choose $j_0$ large such that $t_0+t_{1,j}+t_2+t_{3,j}$ is $\varepsilon$-close to $T$ for $j\ge j_0$.

Using standard properties of Busemann functions, we get
\begin{align*}
	\beta_{\tilde v_-}(\pi(\tilde v), \gamma_0^{j}  \gamma ^ {-i} \pi(\tilde u_{i,j})) =& \beta_{\tilde v_-}(\pi(\tilde v),  \pi(g_{-il(v)}  \gamma_0^{j} \tilde v_{j})) = 
	\beta_{\tilde v_-}(\pi(\tilde v), \gamma_0^{j} \pi(\tilde v_{j})) +il(v) 
	\\
	=&t_0+t_{1,j}+ il(v) +\beta_{\tilde v_-}(\pi(\tilde u), \gamma_0^{j} \pi(\tilde u))
	\\
	=&	t_0+t_{1,j}+ il(v)  + 2 \langle \tilde v_- , \tilde u_+ \rangle_{\pi(\tilde u )} - 2\langle \tilde v_- , \tilde u_+ \rangle_{\gamma_0^{j} \pi(\tilde u)} 
	\\
	&-\beta_{\tilde u_+}(\pi(\tilde u), \gamma_0^{j} \pi(\tilde u))
	\\	
	=&t_0+t_{1,j}+ t_2 + t_{3,j} + il(v) -jl(u).
\end{align*}

For $k\in \naturals $, we put $j_k=j_0+km $ and $i_k=kn$. We write $T'=T-j_0 l(u)$. There exists $-\varepsilon <\alpha_k<\varepsilon$ such that $g_{\alpha_k}\gamma_0^{j_k} \gamma^{-j_k} \tilde u_{i_k,j_k}$ is in the horocycle $H^ u(g_{T'+k\varepsilon_0}\tilde v)$.

We project $\tilde u_{i_k,j_k}$ to vectors $u_{i_k,j_k}$ in the quotient $T^ 1M$. We know that they are in the the positive half-horocycle $h_{\reals_+}(v)$ of $v$ and also in the weak unstable leaf of $v$. Recall that we chose $\tilde v$ so that its geodesic does not bound a flat strip in $X_{\tilde v}^ +$. By Lemma \ref{lemma_Knieper}, the positive unstable half-horocycle of $v$ contracts in the past, so we can find $r\in \naturals$, depending on $k$, such that the distance between $g_{-rl(v)+\alpha_k} u_{i_k,j_k}\in H^ u(g_{T'+k\varepsilon_0} v)$ and $g_{T'+k\varepsilon_0} v$ is less than $\varepsilon$. We deduce that $g_{-rl(v)} u_{i_k,j_k}\in h_{\reals}  v$ is at distance from $g_{T'+k\varepsilon_0} v$ less than $2\varepsilon$. Since this holds for all $k\in \naturals$, we have shown that any vector in the $g_t$-orbit of $v$ is $3\varepsilon$-close to a vector in $h_{\reals_+}(v)$.

\textit{Conclusion.} Putting together the previous steps, we have
$$h_{\reals_+} (\Omega_{NF}) \subset h_{\reals_+} \overline{g_{\reals} h_{\reals_+}(v)} = \overline{ h_{\reals_+} g_{\reals} (v)} \subset \overline{h_{\reals_+}(v)},$$ which combined with Lemma \ref{lemma_Schapira} gives the conclusion.

\end{proof}

\begin{myproof}{Lemma}{\ref{lemma_nonarith}}
	If $\beta_k/\alpha$ is irrational for some $k$, the inequality is just an application of the Dirichlet's approximation theorem.
	
	Assume that all the quotients $\beta_k/\alpha$ are rational. We can find coprime natural numbers $p_k$ and $q_k$ such that $\beta_k/\alpha=p_k/q_k$. First, we observe that $q_k$ must be unbounded. Indeed, if they were bounded by some $q$, the numbers $\beta_k/\alpha$ would be in $1/Q \naturals$ with $Q=\prod_{i=1}^{q}i$. But this would imply that $\beta_k-\beta_{k-1}=\alpha$ for large $k$, contradicting the hypothesis. 
	
	We fix $k$ such that $\alpha/q_k<\varepsilon$. By Bézout's identity, there exist $n,m\in \naturals$ such that $|mq_k-np_k|=1$. Therefore $|m\alpha-n\beta_k|=\alpha/q_k<\varepsilon$.
\end{myproof}

Using the terminology of \cite{Schapira11}, we next prove that a radial point which is not the second end point of an interval is right horocyclic. Given $\xi\in \partial X$, $p\in X$, and $R>0$, we consider the set of points in the horoball $B_\xi(p)$ which are at distance greater that $R$ from the geodesic joining $p$ to $\xi$. This set has two connected components. We denote the connected component which is in the positive half-plane bounded by the geodesic from $p$ to $\xi$ by $B_\xi^ +(p,R)$.

\begin{lemma}\label{right_horocyclic}
	Let $\xi$ be a limit point which and is not the second end point of an interval of $\partial X\setminus \Lambda$. Assume that $\xi $ is not the center of a parabolic isometry. Fix $o\in X$. For every $p\in X$ and every $R>0$, there exists $\gamma\in \Gamma$ such that $\gamma o\in B_\xi^ +(p,R)$.
\end{lemma}

\begin{proof}
	The proof goes by contraposition. Suppose that, for some $p\in X$ and some $R>0$, $\Gamma o$ does not intersect $B_\xi^ +(p,R)$. Let $d$ be the diameter of the compact part $K$ of $M$. Then $\Pi^{-1}(K)$ does not intersect the set $B_\xi^ +(c_{p,\xi}(d),R+d)$. So $B_\xi^ +(c_{p,\xi}(d),R+d)$ must be contained in a connected component $Y_0$ of the complement in $X$ of the set $\Pi^{-1}(K)$ which is a lift of a neighborhood $E_i$ of an end. This implies that the end point $\xi$ is in $\overline{Y_0}$. If $E_i$ is simple parabolic, then $\xi $ is a simple parabolic limit point.
	
	If $E_i$ is cylindrical, expanding or exceptional parabolic, then $\xi$ is necessarily the first end point of the interval $I=\overline{Y_0}\cap \partial X$. The first two cases are impossible, since the set $B_\xi^ +(c_{p,\xi}(d),R+d)$ would not be contained in $Y_0$. The case that $E_i$ is exceptional parabolic and $\xi$ is not the center of the horoball $Y_0$ is not possible for the same reason. The only remaining case is that $E_i$ is exceptional parabolic and $\xi$ is the center of an exceptional parabolic isometry. 
\end{proof}

\begin{proposition}\label{characterization_dense_half}
	Let $v\in \Omega^+$. Assume that the horocycle of $v$ is not periodic. Then, the half-horocycle $h_{\reals_+}(v)$ is dense in $\Omega^ +$ if and only if, for any lift $\tilde v$ of $v$ to $T^ 1X$, $\tilde v_+$ is not the second end point of an interval in $\partial X\setminus \Lambda $.
\end{proposition}

\begin{proof}
	First, we prove that if $\tilde v_+$ is the second end point of an interval of $\partial X\setminus \Lambda$, the half-horocycle $h_{\reals_+}v$ converges to an end of the surface. We consider the interval $I$ of $\partial X\setminus \Lambda $ bounded by $\tilde v_+$. This interval corresponds to an end of the surface $\varepsilon$. We take take a standard neighborhood $E$ of that end, and consider the connected component $\tilde E$ of the lift $\Pi^ {-1}(E)$ corresponding to the interval $I$. Recall that the boundary of $\tilde E$ consists of a geodesic or a horocycle joining the two end points of $I$. 
	
	Let us treat the case that $\varepsilon$ is expanding. The horocycle $\pi( h_{\reals}\tilde v)$ cuts the geodesic $\tilde E$ orthogonally at a unique point $x$. Let $\tilde v'$ be the vector in $h_{\reals}\tilde v$ with basepoint $x$. The half-horocycle $h_{\reals_+}\tilde v'$ stays in $\tilde E$ and diverges from the geodesic boundary $\partial \tilde E$. On the surface $M$, this means that the half-horocycle $h_{\reals_+}v'$ diverges through the end $\varepsilon$.
	
	Now we deal with the case that $\varepsilon$ is exceptional parabolic or cylindrical. The boundary of $\tilde E$ is a horocycle which projects to a closed one, so its center is different from $\tilde v_+$. Let $x$ be any point in $\partial \tilde E$ and consider the vectors $\tilde v'$ and $\tilde w$ based at $x$ pointing to $\tilde v_+$ and to the center of $\partial \tilde E$, respectively. Since there exists no geodesic joining $\tilde v_+$ and the center of $\partial \tilde E$, which is in $\bar I$, the two horocycles $ \pi (h_{\reals} \tilde v')$ and $\partial \tilde E =\pi(h_{\reals } \tilde w)$ intersect just once. Therefore, $\pi(h_{\reals_+ } \tilde v ')$ is contained in $\tilde E$. By Proposition \ref{distance_horocycles}, it diverges from $\pi(h_{\reals_+ }\tilde w)$. Moreover, $\pi(g_{\reals_+}\tilde v')$ stays in the horoball $\tilde E$ and diverges from $\pi(h_{\reals_-}\tilde w)$. Since $\pi(h_{\reals_+}\tilde v')$ is on the half-plane $X^ +_{\tilde v'}$, $\pi(h_{\reals_+}\tilde v')$ also diverges from $\pi(h_{\reals_-}\tilde w)$. The horocycle $\pi(h_{\reals }\tilde v)$ is equidistant from $\pi(h_{\reals_+ } \tilde v ')$, so we also know that $\pi(h_{\reals_+} \tilde v )$ eventually enters the region $\tilde E$ and diverges from $\partial \tilde E$. This proves that $h_{\reals_+} v $ converges to $\varepsilon$.
	
	Now, we assume that $\tilde v_+$ is not the second end point of an interval in $\partial X\setminus \Lambda $ and we want to show that $h_{\reals_+}v$ is dense in $\Omega^ +$. 
	We can choose a reference point $o$ on a lift of the interior of the compact part of $M$. No geodesic passing through $o$ bounds an end of the manifold. If $c$ is a geodesic starting at $o$, the two intervals $(c(-\infty),c(+\infty))$ and $(c(+\infty),c(-\infty))$ necessarily intersect the limit set $\Lambda$.
	
	By Lemma \ref{right_horocyclic}, there exists a sequence $\gamma_n $ in $\Gamma$ such that $\gamma_n o$ stays in the half-plane $X^ +_{\tilde v}$ and converges to the end point $\tilde v_+$, and $t_n := \beta_{\tilde v_+}( o, \gamma_n o)$ tends to $+\infty$. For $n\in\naturals$, we consider the vector $\tilde w_n$ based on $\gamma_n o$ pointing at $\tilde v_+$. Notice that $\tilde v_n:=g_{-t_n} \tilde w_n$ is in the positive stable horocycle of $v$. We write $\tilde v_n = h_{s_n}(\tilde v)$ for some $s_n\ge 0$.

	The sequence of vectors $\gamma_n^ {-1} \tilde w_n$ belongs to $T^1_o X$. Up to taking a subsequence we can assume that it converges to a vector $u\in T^ 1_o X$. The sequence $\gamma_n ^ {-1} \tilde v_+$ converges to $\tilde u_+$, which is also in the limit set $\Lambda$.  
	
	By the observation made above, the limit set intersects the interval $(\tilde u_-,\tilde u_+)$. 	
	Let $ p$ be a $g_t$-periodic rank $1$ vector on $T^1M$ whose geodesic does not bound an end of the surface $M$. We choose a lift $\tilde p$ of $p$ to $T^1 X$ such that $\tilde p_-$ is in the interval $(\tilde u_-,\tilde u_+)$. By Proposition \ref{characterization_dense_hor_periodic}, the half-horocycle $h_{\reals_+}p$ is dense in $\Omega^+$. The negative end point $\tilde p_-$ can be joined to any other point at infinity. We consider the vector $\tilde z$ in $h_{\reals} \tilde u$ with $\tilde z_-=\tilde p_-$ and the vectors $\tilde w_n'\in h_{\reals}(\gamma_n^ {-1} w_n)$ with $\tilde w_{n-}' =\tilde p_-$. Since the geodesic of $\tilde z$ does not bound a flat strip, the vectors $\tilde w_n'$ tend to $\tilde z$ when $n $ goes to infinity. There exist real numbers $s>0,s_n'\in\reals$ such that $\tilde z=h_s(\tilde u)$, $\tilde w_n'=h_{s_n'}(\gamma_n^ {-1}\tilde w_n)$ and $\lim_n s_n'=s$.
	
		\begin{figure}
		\centering
		\includegraphics{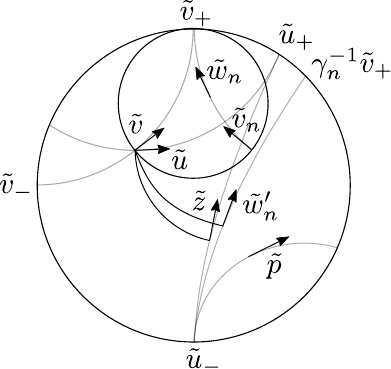}
		\caption{Elements of the proof of Proposition \ref{characterization_dense_half}}
		\label{fig:prop311}
	\end{figure}
	
	Now, since 
	$$\gamma_n \tilde w_n' = h_{s_n'}(\tilde w_n) = h_{s_n'} g_{t_n} \tilde v_n= h_{s_n'} g_{t_n} h_{s_n} \tilde v,$$
	and $s_n'$ is positive for $n$ large enough, we have 
	$$\tilde v_n':=g_{-t_n} \gamma_n \tilde w_n' = g_{-t_n} h_{s_n'} g_{t_n} h_{s_n} \tilde v \in h_{\reals_+}(\tilde v).$$
	On the other hand,
	$$\gamma_n^{-1} \tilde v_n' \in W^ {wu}(\tilde w_n')=W^ {wu}(\tilde p).$$
	Since $\tilde p$ has rank $1$, the weak unstable leaf of $\tilde p$ contracts in the past. 
	Putting $\bar t:=\beta_{\tilde p_-}(\pi( \tilde z), \pi (\tilde p ))$, we have	
	$$d(\gamma_n^ {-1} \tilde v_n', g_{-t_n + \bar t } \, \tilde p)=d(g_{-t_n} \tilde w_n', g_{-t_n + \bar t } \, \tilde p)\xrightarrow[n\to +\infty]{} 0.$$
	
	In the quotient $T^1 M$, this means that the projections $v_n'\in h_{\reals_+}(v)$ of $\tilde v_n'$ approach the periodic orbit of $p$. So they accumulate to a $g_t$-periodic vector with dense positive half-horocycle. This finishes the proof.
	
\end{proof}

\subsubsection{Half-horocycles and collared ends}

Consider an interval $I$ in the complement of the limit set $\Lambda$ that corresponds to a cylindrical or exceptional parabolic. Recall that there exists an horocycle with center in the closure of $I$ whose projection to the surface $M$ is a closed curve that bounds a neighborhood of the end associated to $I$. In the next proposition we study horocyles centered at points in $I$ different from  this distinguished point.
\begin{proposition}\label{dense_parabolic_cylin}
	Let $I=(\eta_1,\eta_2)$ be a interval corresponding to a cylindrical or exceptional parabolic end. Let $\eta\in [\eta_1,\eta_2]$ be the center of the isometry associated to the end. Let $v\in T^1 M$ such that, for some lift $\tilde v$ of $v$ to $T^ 1X$, $\tilde v_+\in [\eta_1,\eta_2]$.
	
	\begin{itemize}
	\item 	
	If $\tilde v_+ \in [\eta_1, \eta)$, then the the positive half-horocycle  $h_{\reals_+}(v)$ is dense in $\Omega^+$.
	\item 
	If $ \tilde v_+ \in (\eta, \eta_2]$, then the the positive half-horocycle  $h_{\reals_+}(v)$ is divergent.
	\end{itemize}
\end{proposition}

\begin{proof}
	Let $H$ be a horocycle in $X$ centered at $\eta $ whose projection to $M$ bounds a neighborhood $E_i$ of the end associated to $I$. The fact that a half-horocycle be dense in $\Omega^+$ or divergent does only depend on its center. Without loss of generality, we can assume that the basepoint $x$ of $\tilde v$ is in $H$.
	Let $H'$ be the horocycle centered at $\tilde v_+$ passing through $x$. The angle at $x$ from $H$ to $H'$ is positive if $ \tilde v_+ \in (\eta, \eta_2]$ and negative if $\tilde v_+ \in [\eta_1, \eta)$. Since $H$ and $H'$ only intersect at $x$ by Proposition \ref{joining}, the half-horocycle $\pi(h_{\reals_+}v)$ stays inside the horoball bounded by $H$ if $ \tilde v_+ \in (\eta, \eta_2]$ and stays outside of it if $\tilde v_+ \in [\eta_1, \eta)$. Moreover, by Proposition \ref{distance_horocycles}, the distance between $\pi(h_s(v))$ and $H$ goes to infinity when $s\to+\infty$. This shows that the half-horocycle $h_{\reals_+}v$ is divergent when $ \tilde v_+ \in (\eta, \eta_2]$. 
	
	We assume that $ \tilde v_+ \in (\eta, \eta_2]$ and show that the half-horocycle is dense. 	
	Let $\gamma$ be the exceptional parabolic or axial isometry associated to $I$. This isometry fixes $H$, and up to taking the inverse, we can assume that $\gamma^ n x$ tends to $\eta_1$ when $n$ goes to infinity. For $n\in \naturals$, consider the vector $w_n$ with basepoint $\gamma^ nx$ pointing to $\tilde v_+$. Observe that $\tilde w_n =\gamma^ n\tilde v$.
	We put $t_n:=\beta_{\tilde v_+}(x,\gamma^ n x)$ and consider the vectors $\tilde v_n:=g_{-t_n} \tilde w_n\in h_{\reals} \tilde v$. Moreover, since $\gamma$ fixes $\tilde v_+$,
	$$t_n= n\beta_{\tilde v_+}(x,\gamma x),$$
	where $\beta_{\tilde v_+}(x,\gamma x)$ is strictly positive.
	Also, $\gamma^n x$ is in the half plane $X^ +_{\tilde v}$, so $\tilde v_n$ is in the positive half-orbit $h_{\reals_+}\tilde v$. We remark that $\gamma^ {-n} \tilde v_n=g_{-t_n} (\tilde v)$.
	
	Let $ p$ be a $g_t$-periodic rank $1$ vector on $T^1M$ whose geodesic does not bound an end of the surface $M$. We choose a lift $\tilde p$ of $p$ to $T^1 X$ such that $\tilde p_-$ is in the interval $(\tilde v_-,\tilde v_+)$. By Proposition \ref{characterization_dense_hor_periodic}, the half-horocycle $h_{\reals_+}p$ is dense in $\Omega^+$. The negative end point $\tilde p_-$ can be joined to any other point at infinity. We consider the vector $\tilde z$ in $h_{\reals} \tilde v$ with $\tilde z_-=\tilde p_-$. For every $n\in \naturals$, there exists $s_n\ge 0$, such that $h_{s_n}(\gamma^ {-n} \tilde v_n)=g_{-t_n}(\tilde z)$. As in the proof of Proposition \ref{characterization_dense_half}, the projection $v_n$ of $\tilde v_n$ to $T^1 M$ tends to a vector in the periodic orbit of $p$. Since the positive half-orbit of $p$ is dense in $\Omega^ +$, we get the desired result.

\end{proof}

\subsection{Closure of horocyclic orbits}

\begin{proposition}\label{noaccpoints}
	Any accumulation point of an orbit $\Gamma\xi$, $\xi\in \partial X$, is in the limit set $\Lambda$.
\end{proposition}
\begin{proof}
	We reason by contradiction. Assume that there exists a sequence $(\gamma _n)_n $ of pairwise distinct isometries in $\Gamma$ such that $\gamma_n \xi$ converges to $\eta\in \partial X\setminus \Lambda$. The point $\eta $ belongs to an open interval $I$ of $ \partial X\setminus\Lambda$. There exists an $n_0$ such that, for $n\ge n_0$, $\gamma_n \xi$ belongs to $I$. Since $g_n=\gamma_n \gamma_{n_0}^ {-1}$ sends an element of $I$ to another, it leaves invariant $I$, but is different from the identity by assumption. We observe that $g_n$, $n\ge n_0$, can never be simple parabolic. 
	
	If $g_{n_0}$ is axial, the only way that it leave $I$ invariant is that the extremities of $I$ are the attracting and the repelling fixed points of $g_{n_0}$. So $I$ is associated to an expanding end of the surface. This implies that all the $g_n$ for $n\ge n_0$ are powers of a common axial isometry $\gamma$. The only possible accumulation points of $\langle \gamma \rangle \gamma_{n_0}\xi$ are the extremities of $I$, which are in the limit set.
	
	If $g_{n_0}$ is exceptional parabolic, the only way that it leave $I$ invariant is that $I$ is equal to the set of fixed points of $g_{n_0}$. So $I$ is associated to an exceptional parabolic end of the surface. This implies that all the $g_n$ for $n\ge n_0$ are powers of a common exceptional parabolic isometry $\gamma$ which fixes $I$. Therefore, the sequence $\gamma_n \xi=g_n \gamma_{n_0}\xi$ is constant for $n\ge n_0$ and equal to $\eta$. Thus $\eta$ is not an accumulation point.

\end{proof}

As a consequence of Proposition \ref{noaccpoints}, there is a unique minimal subset for the action of $\Gamma$ on $\partial X$.

\begin{proposition}\label{accumulation}
	Any accumulation point of an orbit $h_{\reals}v$ is in $\Omega^+$.
\end{proposition}

\begin{proof}
	It is an application of the duality between the action of the horocyclic flow on $T^ 1M$ and the action of $\Gamma $ on $\partial X$.
\end{proof}

\begin{definition}
	A geodesic ray $c: \reals_+ \to M$ is \emph{eventually minimizing} if there exists $t_0$ such that for $t,t'\ge t_0$, $$d(c(t'),c(t))=|t'-t|.$$ 
\end{definition}

We notice that an eventually minimizing geodesic ray converges to an end of the manifold. If the end is parabolic or cylindrical, the geodesic is perpendicular to the periodic horocycles that bound the end.

Now we can finally prove the main result of the article.

\begin{myproof}{Theorem}{\ref{main_theorem}}
		Let $g_{\reals_+}v$ be a geodesic ray that converges to an end $ \varepsilon$. Consider a parametrization in coordinates $(r,\theta)$ of a neighborhood of $\varepsilon$ given by Theorems \ref{coordinates-nbgh} and \ref{classification_ends}. First, we assume that the end is parabolic or cylindrical. The only eventually minimizing geodesic rays that converge to $\varepsilon$ are those that when they enter the neighborhood of $\varepsilon$ are $r$-curves. Any horocycle associated to a vector tangent to such curves is periodic. 
		
		The other geodesic rays that converge to $\varepsilon$ are not eventually minimizing. By Proposition \ref{dense_parabolic_cylin}, one half-horocycle associated to vectors tangent to these rays is dense in $\Omega^+$. Thus, we have the inclusion $ h_{\reals}v \cup \Omega^ +\subset \overline{h_{\reals}v}  $. Proposition \ref{accumulation} gives the other inclusion.
		
		If the end is expanding, then by Proposition \ref{TypesCC} the horocycle of $v$ converges to the end $\varepsilon$.
		
		Finally, if $g_{\reals_+}v$ does not converge to any end, then any lift of $v$ to $T^ 1X$ points to a non parabolic limit point. In particular $v$ belongs to $\Omega^ +$. By Proposition \ref{Geometrically_finite_limit_points}, the limit point is radial and by Proposition \ref{characterization_dense_half}, at least one half-horocycle of $v$ is dense in $\Omega^+$.
\end{myproof}

\subsection{Minimal subsets of the horocyclic flow}

We now seek $h_s$-invariant closed nonempty subsets of $T^1 M $ which are minimal with this property. Periodic noncompact closed orbits of $h_s$ are the trivial examples of such sets.
We obtain Corollary \ref{minimal_subsets}, which asserts that there is at most one nontrivial minimal subset and gives a necessary and sufficient condition for its existence.

\begin{myproof}{Corollary}{\ref{minimal_subsets}}
	The set $\Omega^+$ is always $h_s$-invariant, closed and nonempty. According to Proposition \ref{Geometrically_finite_limit_points} and Lemma \ref{intersection_types_limit}, saying that all the points $\Lambda$ are horocyclic is equivalent to asking that $\Lambda$ contains no center of a parabolic isometry. If $\Lambda$ does not contain centers of parabolic isometries, then no $v$ in $\Omega^+$ has a periodic horocycle. By Proposition \ref{characterization_dense_half}, every $v$ in $\Omega^+$ has dense orbit in $\Omega^+$, so $\Omega^+$ is minimal.
	
	For the converse, a minimal subset is equal to the closure of the orbit of any of its elements. Assume that $\overline{ h_{\reals}  v}$ is minimal but $h_{\reals}  v$ is not closed. Then $ h_{\reals}  v$ accumulates somewhere in $\Omega^+$, so in fact $\Omega^+$ is contained in $\overline {h_{\reals}  v}$. Since the latter is minimal, they must be equal. But for $\Omega^+$ to be minimal, there cannot be periodic horocycles in $\Omega^+$, put otherwise, $\Lambda$ has no centers of parabolic isometries.
\end{myproof}

The condition of the corollary amounts to ask that $M$ has neither simple parabolic ends nor exceptional parabolic ends with center in the limit set. We do not know if the latter can occur. 

\subsection{Nonwandering set of the horocyclic flow}
 In this section, we prove that the nonwandering set of the horocyclic flow is
	$$\NW(h_s)=\Omega^+\cup \Per (h_s).$$

\begin{myproof}{Theorem}{\ref{nonwandering_set}}
	First, we prove the inclusion to the left. Periodic vectors are always nonwandering and nonperiodic vectors in $\Omega^+$ have dense orbits in $\Omega^+$, which implies that they are also nonwandering (in fact, recurrent) because the $h_s$ orbits of $\Omega^+$ are not isolated. We can see this as follows. Let $v\in \Omega^+$ be nonperiodic and let $U$ be a neighborhood of $v$ in $\Omega^+$. For any $R>0$, the set $U\setminus h_{[-R,R]}(v)$ is always nonempty, for example, it contains $g_t(v)$ for $t$ small. By the density of $h_{\reals}  v$, there exists $s\in \reals $ such that $h_s(v)$ is in $U\setminus h_{[-R,R]}(v)$, so $|s|>R$. So $h_{\reals}  v$ intersects every neighborhood of $v$ with arbitrarily large times.
	
	Now, let us show that a vector $v$ which is not in $\Omega^+\cup \Per (h_s)$ is wandering. We take a lift $\tilde v$ to $T^1 X$. Its end point $\tilde v_+$ is outside the limit set and is not fixed by an exceptional parabolic isometry. Let $I$ be the interval of $\partial X \setminus \Lambda$ that contains $\tilde v_+$ and let $\varepsilon$ be the end associated to it. Let $\tilde E$ be the lift of a standard neighborhood of $\varepsilon$, which consists of a horoball or a half-plane. Its boundary $\partial \tilde E$ is a geodesic or a horocycle.

	We distinguish two cases. If the end $\varepsilon$ is expanding, we consider any open interval $J$ containing $\tilde v_+$ relatively compact in $I$.
	If the end is exceptional parabolic or cylindrical, we also consider an open interval $J$ containing $\tilde v_+$ relatively compact in $I$, but we ask in addition that $\overline J$ does not contain the center of the horoball $\tilde E$. 
	
	Given a small $\varepsilon>0$, we consider an open neighborhood of $\tilde v$ of the form 
	\[
	V = \{ \tilde w\in T ^1 X | \pi(\tilde w) \in B(\pi(v), \varepsilon) , \, \tilde w_+\in J \}.
	\] 
	First we deal with the parabolic or cylindrical case. By Proposition \ref{distance_horocycles}, the distance between the horocycle $\partial \tilde E $ and $\pi(h_s(v)))$ goes to infinity for both $s\to \pm \infty$. Since $\tilde v_+\in I$, for positive or negative large time $s$, $\pi(h_s(\tilde v))$ is in the horoball $\tilde E$ and stays inside for larger times. We denote the sign of the time $s$ by $\delta \in \{-1,1\}$. In fact, the same result implies the following property: for every $\tilde w\in \overline{V}$, there exists $s_0>0$ such that for every $s\in \reals$ such that $\delta s\ge s_0$, $h_s(\tilde w)$ is based on $\tilde E$ and 
	$$
	\lim_{s\to \delta \cdot \infty} d(\partial \tilde E , \pi(h_s(\tilde w)))=+\infty.
	$$
	This property is also true when $\varepsilon$ is an expanding end for both choices of the sign $\delta$ by Proposition \ref{TypesCC}.
	Since $\tilde V$ is compact, we can find $s_0$ such that for every $s$ with $\delta s\ge s_0$, $\pi(h_s(\overline V))\subset \tilde E$ and 
	$$
	\inf_{\tilde w\in h_s(\overline V)} d(\partial \tilde E, \tilde w) > \sup_{\tilde w \in \overline V } d ( \Gamma \partial \tilde E, \tilde w).
	$$
	This guarantees that if we take the projection $U$ of $V$ to the quotient $T^1 M$, for $\delta s\ge s_0$, $h_s(U) $ does not intersect $U$. This says that the vector $v$ is wandering.

\end{myproof}

\bibliographystyle{alpha}
\bibliography{general_library}
\end{document}